\documentclass{amsart}%
\usepackage{amsthm}
\usepackage{amsmath}
\usepackage{amsfonts}
\usepackage{amssymb}
\usepackage{dsfont}
\usepackage{esint}
\usepackage{graphicx}
\usepackage{color}
\usepackage{url}
\usepackage[latin1]{inputenc}
\usepackage[T1]{fontenc}
\usepackage[english]{babel}%
\usepackage{amsaddr}

\providecommand{\U}[1]{\protect\rule{.1in}{.1in}}

\newtheorem{theorem}{Theorem}[section]

\newtheorem{definition}{Definition}[section]

\newtheorem{lemma}{Lemma}[section]

\newtheorem{proposition}{Proposition}[section]
\newtheorem{remark}{Remark}[section]

\newcommand{\rst}[1]{\ensuremath{\raise-1.0ex\hbox{\large{$\vert_{#1}$}}}}

\begin{document}

\title[Existence for the generalized Navier-Stokes equations with damping]{Existence of weak solutions for the generalized Navier-Stokes equations with damping}

\date{\textbf{September 23, 2011}}

\author[H.B. de Oliveira]{H.B. de Oliveira$^{\ast,\ast\ast}$}

\email{holivei@ualg.pt}

\address{$^{\ast}$FCT - Universidade do Algarve,
Campus de Gambelas,
8005-139 Faro, Portugal}
\address{$^{\ast\ast}$CMAF - Universidade de Lisboa,
Av. Prof. Gama Pinto, 2,
1649-003 Lisboa, Portugal.}

\thanks{The author's work was supported by the grant SFRH/BSAB/1058/2010, MCTES, Portugal, and by the research project PTDC/MAT/110613/2010, FCT, Portugal.}

\begin{abstract}
In this work we consider the generalized Navier-Stoke equations with the presence of a damping term in the momentum equation. %
The problem studied here derives from the set of equations which govern the isothermal flow of incompressible, homogeneous and non-Newtonian fluids. %
For the generalized Navier-Stokes problem with damping, we prove the existence of weak solutions by using regularization techniques, the theory of monotone operators and compactness arguments together with the local decomposition of the pressure and the Lipschitz-truncation method.
The existence result proved here holds for any $q>\frac{2N}{N+2}$ and any $\sigma>1$, where $q$ is the exponent of the diffusion term and $\sigma$ is the exponent which characterizes the damping term.



\end{abstract}

\maketitle

\selectlanguage{english}


%\author{}

%\date{}

%\date{}
%\dedicatory{To my parents}
%%% ----------------------------------------------------------------------

\section{Introduction}\label{Sect-Int}

In this work, we shall study the existence of weak solutions for the generalized Navier-Stokes equations with damping:
\begin{equation}\label{geq1-inc}
\mathrm{div}\,\mathbf{u}=0\quad\mbox{in}\quad Q_T,
\end{equation}
\begin{equation}\label{geq1-vel}
\frac{\partial\,\mathbf{u}}{\partial\,t}
-\mathbf{div}\left(|\mathbf{\nabla}\mathbf{u}|^{q-2}\mathbf{\nabla}\mathbf{u}-
\mathbf{u}\otimes\mathbf{u}\right)
+\alpha|\mathbf{u}|^{\sigma-2}\mathbf{u}
=\mathbf{f}-\mathbf{\nabla}p
 \quad\mbox{in}\quad Q_T;
\end{equation}
supplemented with the following initial and boundary conditions:
\begin{equation}\label{geq1-ic}
\mathbf{u}=\mathbf{u}_0\qquad\mbox{in}\quad\Omega \quad\mbox{for}\quad t=0,
\end{equation}
\begin{equation}\label{geq1-bc}
\mathbf{u}=\mathbf{0} \qquad\mbox{on}\quad \Gamma_T.
\end{equation}
Here $Q_T$ is a general cylinder defined by
$$Q_T:=\Omega\times(0,T),\quad\mbox{with}\quad \Gamma_T:=\partial\Omega\times(0,T),$$
where
$\Omega\subset\mathds{R}^N$, $N\geq 2$, is a bounded domain with a compact boundary $\partial\Omega$, and $0<T<\infty$. %

In the scope of Mathematical Fluid Mechanics, $\mathbf{u}$ is the velocity field, $p$ stands for the pressure divided by
the constant density, $\mathbf{f}$ is the given forcing term and $q>1$ is the constant exponent which characterizes the flow. %
The constant $\alpha$ is non-negative and $\sigma>1$ is another constant. %

The damping term $\alpha|\mathbf{u}|^{\sigma-2}\mathbf{u}$, or sometimes called absorption term, has no direct physical justification in Fluid Mechanics, although it might be considered has being part of the external body forces field (see~\cite{CRM-2002}--\cite{RLMA-2004}). %
There is also a precise theory of the absorption of forced plane infinitesimal waves according to the Navier-Stokes equations (see~\cite{True-1953}).
The consideration of damping terms in the generalized Navier-Stokes equations it is also useful as a regularization procedure to prove the existence of weak solutions for the stationary problems (see \cite{FMS-1997}--\cite{FMS-2003}).
At last, but not in last, there is also the purely mathematical motivation which goes back to a work about a stationary like problem (see~\cite{BBC-1975}), where the authors where mainly interested with the important question about compact supported solutions for that problem. %
During the last years, many authors have worked on these kind of modified Navier-Stokes type problems, establishing the existence of weak solutions and proving many other properties has the uniqueness of weak solutions, their regularity and studying its asymptotic behavior. %
In~\cite{AMFM-2010} we proved the weak solutions of (\ref{geq1-inc})-(\ref{geq1-bc}) extinct in a finite time for $q\geq 2$, provided $1<\sigma<2$. This property is well known for the generalized Navier-Stokes problem  (\ref{geq1-inc})-(\ref{geq1-bc}) with $\alpha=0$ in the case $1<q<2$. But for $q\geq 2$ the best one can gets are some decays of fractional and exponential order (see \emph{e.g.}~\cite{Bae-1999}). In~\cite{AA-2010} we have studied the problem (\ref{geq1-inc})-(\ref{geq1-bc}) in the particular case of $q=2$. There, we have proved the existence of weak solutions, its uniqueness and some asymptotic properties. We carried out an analogous study in~\cite{JMAA-2011} for the Oberbeck-Boussinesq version of this problem, where besides the usual coupling in the buoyancy force, we have considered an extra coupling in the damping term by considering a temperature-depending function $\sigma$. In~\cite{CJ-2008} the authors have proved the existence of weak and strong solutions for the Cauchy problem (\ref{geq1-inc})-(\ref{geq1-bc}) in $\mathds{R}^3$ and with $q=2$. The damping term is being considered in the context of many other physical systems which go from the Shrödinger equations (see \emph{e.g.}~\cite{CG-2011})  to the Euler equations (see \emph{e.g.}~\cite{PZ-2008}) and passing by the wave equation (see \emph{e.g.}~\cite{Zhou-2005}).

With respect to the existence of weak solutions for the original generalized Navier-Stokes problem, \emph{i.e.} (\ref{geq1-inc})-(\ref{geq1-bc}) with $\alpha=0$, the problem was solved in its all full possible (it is open only the case $1<q<\frac{2N}{N+2}$) extension recently in the work \cite{DRW-2010}.
The first existence result to this problem was achieved in~\cite{L-1967} and~\cite{Lions-1969} for $q\geq \frac{3N+2}{N+2}$.
Only more or less 40 years later it was possible to improve the existence result for lower values of $q$.
In~\cite{Zhikov-2009}, under the same assumptions of~\cite{L-1967} and~\cite{Lions-1969}, it was proved the existence of weak solutions to the problem (\ref{geq1-inc})-(\ref{geq1-bc}) with $\alpha=0$ for $q\geq \max\left\{\frac{3N}{N+2},\frac{N+\sqrt{3N^2+4N}}{N+2}\right\}$.
A lit bit earlier to the work~\cite{Zhikov-2009}, it was proved in~\cite{Wolf-2007} the existence of a weak solution to the same problem for $q>2\frac{N+1}{N+2}$.
Finally in~\cite{DRW-2010} the authors have extended the result~\cite{Wolf-2007} to the case $q>\frac{2N}{N+2}$.
It is an open problem to prove the existence of weak solutions to the problem (\ref{geq1-inc})-(\ref{geq1-bc}) with $\alpha=0$ if
$1< q\leq \frac{2N}{N+2}$ in the case of $N>2$.
On the other hand, it seems to be very difficult to go bellow the limit $q=\frac{2N}{N+2}$ (for $N>2$), due to the need of using the compact imbedding $\mathbf{W}^{1,q}(\Omega)\hookrightarrow \mathbf{L}^2(\Omega)$.

The plan of this work is the following.
In Section~\ref{Sect-Int} we introduce the problem we shall study here and review some results related with our work.
Section~\ref{Sect-WF} is devoted to introduce the notation we use throughout the work and to define the notion of weak solution we shall consider.
Here, we also shall state the main result of this paper: Theorem~\ref{th-exst-wsqv}. The proof of this result is carried out from Section~\ref{Sect-ERP} to Section~\ref{Sect-Conv-EST}. In Section~\ref{Sect-Rem}, we make some remarks about our work, in special, its extensions and limitations.

\section{Weak formulation}\label{Sect-WF}

The notation used throughout this article is largely
standard in  Mathematical Fluid Mechanics - see \emph{e.g.}~\cite{Lions-1969}. %
We distinguish tensors and vectors from scalars by using boldface letters. For functions and function spaces we will
use this distinction as well. %
The symbol $C$ will denote a generic constant - generally a positive one, whose value will not
be specified; it can change from one inequality to another. %
The dependence of $C$ on other constants or parameters will always be clear from the exposition. %
In this paper, the notations $\Omega$ or $\omega$ stand always for a domain, \emph{i.e.}, a connected open subset of $\mathds{R}^N$, $N\geq 1$.
By $\mathcal{L}_N(\omega)$ we denote the $N$-dimensional Lesbesgue measure of $\omega$.
Given $k\in\mathds{N}$, we denote by $\mathrm{C}^{k}(\Omega)$ the space of all $k$-differentiable functions in $\Omega$. By $\mathrm{C}^{\infty}_0(\Omega)$ or $\mathcal{D}(\Omega)$, we denote the space of all infinity-differentiable functions with compact support in $\Omega$. The space of distributions over $\mathcal{D}(\Omega)$ is denoted by $\mathcal{D}'(\Omega)$. %
If $\mathrm{X}$ is a generic Banach space, its dual space is denoted by $\mathrm{X}'$.
Let $1\leq q\leq \infty$ and $\Omega\subset\mathds{R}^N$, with $N\geq 1$, be a domain. %
We will use the classical Lebesgue spaces $\mathrm{L}^q(\Omega)$, whose norm is
denoted by $\|\cdot\|_{\mathrm{L}^q(\Omega)}$. %
For any nonnegative $k$,
$\mathrm{W}^{k,q}(\Omega)$ denotes the Sobolev space of all
functions $u\in\mathrm{L}^q(\Omega)$ such
that the weak derivatives $\mathrm{D}^{\alpha}u$ exist, in the generalized sense, and are in
$\mathrm{L}^q(\Omega)$ for any multi-index $\alpha$ such that
$0\leq |\alpha|\leq k$. The norm in $\mathrm{W}^{k,q}(\Omega)$ is denoted by
$\|\cdot\|_{\mathrm{W}^{k,q}(\Omega)}$.
The corresponding spaces of vector-valued or tensor-valued
functions are denoted by boldface letters. All these spaces are
Banach spaces and the Hilbert framework corresponds to $q=2$.
In the last case, we use the abbreviation $\mathrm{W}^{k,2}=\mathrm{H}^{k}$. %
Given $T>0$ and a Banach space $X$,
$\mathrm{L}^q(0,T;X)$ and $\mathrm{W}^{k,q}(0,T;X)$ denote
the usual Bochner spaces used in evolutive problems,
with norms denoted by $\|\cdot\|_{\mathrm{L}^q(0,T;X)}$ and
$\|\cdot\|_{\mathrm{W}^{k,q}(0,T;X)}$. %
By $\mathrm{C}_{\mathrm{w}}([0,T];X)$ we denote the subspace of $\mathrm{L}^{\infty}(0,T;X)$ consisting of functions which are weakly continuous
from $[0,T]$ into $X$.

A very important property satisfied by the tensor $\mathbf{S}=|\mathbf{\nabla\,u}|^{q-2}\mathbf{\nabla\,u}$ and by the damping term $|\mathbf{u}|^{\sigma-2}\mathbf{u}$ are expressed in the following lemma which proof we address the reader to~\cite{BL-1994}.
\begin{lemma}
For all $s\in(1,\infty)$ and $\delta\geq 0$, there exist constants $C_1$ and $C_2$, depending on $s$ and $N$,
such that for all $\textbf{$\xi$},\ \textbf{$\eta$}\in\mathds{R}^N$, $N\geq 1$,
\begin{equation}\label{gin-xi-eta}
\left(|\textbf{$\xi$}|^{s-2}\textbf{$\xi$}-|\textbf{$\eta$}|^{s-2}\textbf{$\eta$}\right)\cdot(\textbf{$\xi$}-\textbf{$\eta$})\geq
C_2|\textbf{$\xi$}-\textbf{$\eta$}|^{2+\delta}\left(|\textbf{$\xi$}|+|\textbf{$\eta$}|\right)^{s-2-\delta}\,.
\end{equation}
\end{lemma}

In order to define the notion of weak solutions we shall look for, let us introduce the usual functional setting of Mathematical Fluid Mechanics:
\begin{equation}\label{V-reg}
\mathcal{V}:=\{\mathbf{v}\in\mathbf{C}_0^{\infty}(\Omega):\mathrm{div\,}\mathbf{v}=0\}\,;
\end{equation}
\begin{equation}\label{H-div}
\mathbf{H}:=\mbox{closure of $\mathcal{V}$ in $\mathbf{L}^2(\Omega)$}\,;
\end{equation}
\begin{equation}\label{V-q=const}
\mathbf{V}_q:=\mbox{closure of $\mathcal{V}$ in $\mathbf{W}^{1,q}(\Omega)$}\,.
\end{equation}
The weak solutions we are interested in are usually called in the sense of Leray-Hopf.

\begin{definition}\label{weak-sol-vq}
Let $N \geq 2$ and $1<q,\ \sigma<\infty$. %
Assume that $\mathbf{u}_0\in \mathbf{H}$, $\mathbf{f}\in\mathbf{L}^{1}(Q_T)$ and let conditions (A)-(D) be fulfilled for any $q>1$.
A vector field $\mathbf{u}$ is a weak solution to the problem (\ref{geq1-inc})-(\ref{geq1-bc}), if:
\begin{enumerate}
\item $\mathbf{\mathbf{u}}\in  \mathrm{L}^{\infty}(0,T;\mathbf{H})\cap \mathrm{L}^q(0,T;\mathbf{V}_q)\cap\mathbf{L}^{\sigma}(Q_T)$;
\item For every $\varphi\in\mathbf{C}^{\infty}(Q_T)$, with $\mathrm{div\,}\varphi=0$ and $\mathrm{supp\,}\varphi\subset\subset\Omega\times[0,T)$,
\begin{equation*}
\begin{split}
& -\int_{Q_T}\mathbf{u}\cdot\varphi_t\,d\mathbf{x}dt+
\int_{Q_T}\left(|\mathbf{\nabla}\mathbf{u}|^{q-2}\mathbf{\nabla}\mathbf{u}-\mathbf{u}\otimes\mathbf{u}\right):\mathbf{\nabla\varphi}\,d\mathbf{x}dt+
\alpha\int_{Q_T}|\mathbf{u}|^{\sigma-2}\mathbf{u}\cdot\varphi\,d\mathbf{x}dt
 \\
& =\int_{Q_T}\mathbf{f}\cdot\varphi\,d\,\mathbf{x}dt+\int_{\Omega}\mathbf{u}_0\cdot\varphi(0)\,d\mathbf{x}.
\end{split}
\end{equation*}
\end{enumerate}
\end{definition}

\begin{remark}
Note that $\mathbf{V}_q\hookrightarrow\mathbf{L}^{\sigma}(\Omega)$ for $\sigma\leq q^{\ast}$, where $q^{\ast}$ is the Sobolev conjugate of $q$, \emph{i.e.} $q^{\ast}=\frac{Nq}{N-q}$ if $1<q<N$, or $q^{\ast}=\infty$ if $q\geq N$.
Therefore, in this case, we look for weak solutions in the class $\mathrm{L}^{\infty}(0,T;\mathbf{H})\cap L^q(0,T;\mathbf{V}_q)$.
\end{remark}

\noindent The main result of this work is the following, where it is established the existence of weak solutions to the problem (\ref{geq1-inc})-(\ref{geq1-bc}) under the minor possible assumptions on $q$ and $\sigma$. We left open only the case of $1<q\leq\frac{2N}{N+2}$, for $N>2$, which will certainly require a different approach.

\begin{theorem}\label{th-exst-wsqv}
Let $\Omega$ be a bounded domain in $\mathds{R}^{N}$, $N\geq 2$, with a Lipschitz-continuous boundary $\partial\Omega$. %
Assume that
\begin{equation}\label{h1-f}
\mathbf{f}\in\mathbf{L}^{q'}(0,T;\mathbf{V}_q'),
\end{equation}
\begin{equation}\label{eq1-vel-u0}
\mathbf{u}_0\in \mathbf{H}.
\end{equation}
Then, if
\begin{equation}\label{exi-c-q}
q>\frac{2N}{N+2},
\end{equation}
there exists a weak solution to the problem (\ref{geq1-inc})-(\ref{geq1-bc}), in the sense of Definition~\ref{weak-sol-vq}, for any $\sigma>1$.
Moreover, any weak solution $\mathbf{u}\in\mathrm{C}_{\mathrm{w}}([0,T];\mathbf{H})$.
\end{theorem}

\noindent In order to simplify the exposition, we shall assume throughout the rest of this work the following simplified assumption of (\ref{h1-f})
\begin{equation}\label{h-f}
\mathbf{f}=-\mathbf{div}\,\mathbf{F},\quad \mathbf{F}\in\mathbf{L}^{q'}(Q_T).
\end{equation}

\noindent The main ingredients of the proof of Theorem~\ref{th-exst-wsqv} are the results of the local decomposition of the pressure established in \cite{Wolf-2007} and the Lipschitz--truncation method in the spirit of \cite{DRW-2010}. %
With could also have considered the $\mathrm{L}^{\infty}$--truncation method used in \cite{Wolf-2007}, but by this method we cannot achieve an existence result for so lower values of $q$ as we can with the  Lipschitz--truncation method. %
The proof of Theorem~\ref{th-exst-wsqv} will be the aim of the next sections. %

\section{The regularized problem}\label{Sect-ERP}

We start the proof of Theorem~\ref{th-exst-wsqv} by considering a regularization of the problem (\ref{geq1-inc})-(\ref{geq1-bc}) which basically gets rid off the difficulties coming from the convective term $\mathbf{u}\otimes\mathbf{u}$. %
Let $\Phi\in \mathrm{C}^{\infty}([0,\infty))$ be a non-increasing function such that $0\leq\Phi\leq 1$ in $[0,\infty)$, $\Phi\equiv 1$ in $[0,1]$, $\Phi\equiv 0$ in $[2,\infty)$ and $0\leq -\Phi'\leq 2$. %
For $\epsilon>0$, we set
\begin{equation}\label{Phi-e}
\Phi_{\epsilon}(s):=\Phi(\epsilon s),\quad s\in[0,\infty),
\end{equation}
and let us consider the following regularized problem:
\begin{equation}\label{eq2-inc-e}
\mathrm{div}\,\mathbf{u}_{\epsilon}=0\quad\mbox{in}\quad Q_T,
\end{equation}
\begin{equation}\label{eq2-vel-e}
\frac{\partial\,\mathbf{u}_{\epsilon}}{\partial\,t}-
\mathbf{div}\left(|\mathbf{\nabla}\mathbf{u}_{\epsilon}|^{q-2}\mathbf{\nabla}\mathbf{u}_{\epsilon}-\mathbf{u}_{\epsilon}\otimes\mathbf{u}_{\epsilon}\Phi_{\epsilon}(|\mathbf{u}_{\epsilon}|)\right)
+\alpha|\mathbf{u}_{\epsilon}|^{\sigma-2}\mathbf{u}_{\epsilon}
=\mathbf{f}-\mathbf{\nabla}p_{\epsilon},
\quad\mbox{in}\quad Q_T,
\end{equation}
\begin{equation}\label{eq1-ic-u-e}
\mathbf{u}_{\epsilon}=\mathbf{u}_0\qquad\mbox{in}\quad\Omega \quad\mbox{for}\quad t=0,
\end{equation}
\begin{equation}\label{eq1-bc-u-e}
\mathbf{u}_{\epsilon}=\mathbf{0}\qquad\mbox{on}\quad \Gamma_T.
\end{equation}
A vector function $\mathbf{u}_{\epsilon}\in \mathrm{L}^{\infty}(0,T;\mathbf{H})\cap\mathrm{L}^{q}(0,T;\mathbf{V}_{q})\cap \mathbf{L}^{\sigma}(Q_T)$ is a weak solution to the problem (\ref{eq2-inc-e})-(\ref{eq1-bc-u-e}), if
\begin{equation}\label{eq-ws-reg}
\begin{split}
& -\int_{Q_T}\mathbf{u}_{\epsilon}\cdot\varphi_t\,d\mathbf{x}dt+
\int_{Q_T}\left(|\mathbf{\nabla}\mathbf{u}_{\epsilon}|^{q-2}\mathbf{\nabla}\mathbf{u}_{\epsilon}-\mathbf{u}_{\epsilon}\otimes\mathbf{u}_{\epsilon}\Phi_{\epsilon}(|\mathbf{u}_{\epsilon}|)\right):\mathbf{\nabla\varphi}\,d\mathbf{x}dt
 \\
& +\alpha\int_{Q_T}|\mathbf{u}_{\epsilon}|^{\sigma-2}\mathbf{u}_{\epsilon}\cdot\varphi\,d\mathbf{x}dt
=\int_{Q_T}\mathbf{F}:\mathbf{\nabla\varphi}\,d\,\mathbf{x}dt+\int_{\Omega}\mathbf{u}_0\cdot\varphi(0)\,d\mathbf{x}
\end{split}
\end{equation}
for all $\varphi\in\mathbf{C}^{\infty}(Q_T)$, with $\mathrm{div\,}\varphi=0$ and $\mathrm{supp\,}\varphi\subset\subset\Omega\times[0,T)$. %

\begin{proposition}\label{th-uxu-reg}
Let the assumptions of Theorem~\ref{th-exst-wsqv} be fulfilled. %
Then, for each $\epsilon>0$, there exists a weak solution $\mathbf{u}_{\epsilon}\in\mathrm{L}^{q}(0,T;\mathbf{V}_{q})\cap\mathrm{C}_{\mathrm{w}}([0,T];\mathbf{H})\cap\mathbf{L}^{\sigma}(Q_T)$ to the problem (\ref{eq2-inc-e})-(\ref{eq1-bc-u-e}). In addition, every weak solution satisfies to the following energy equality:
\begin{equation}\label{e-equality-qr}
\frac{1}{2}\|\mathbf{u}_{\epsilon}(t)\|_{\mathbf{H}}^2+
\int_{Q_t}|\mathbf{\nabla}\mathbf{u}_{\epsilon}|^{q}d\mathbf{x}dt+
\alpha\int_{Q_t}|\mathbf{u}_{\epsilon}|^{\sigma}d\mathbf{x}dt =\frac{1}{2}\|\mathbf{u}_0\|_{\mathbf{H}}^2+
\int_{Q_t}\mathbf{F}:\nabla\mathbf{u}_{\epsilon}d\mathbf{x}dt
\end{equation}
for all $t\in(0,T)$.
\end{proposition}
\begin{proof} The proof of Proposition~\ref{th-uxu-reg} is adapted from the proof of
\cite[Theorem 3.1]{Wolf-2007}. %
The difference here is the presence of an extra term which results from the damping and the aspect of the diffusion term.
We shall split this proof into three steps.

\vspace{0.2cm}
\noindent
\emph{First Step.}
Let $T_{\ast}\in(0,T]$ be arbitrarily chosen and let us set
$$M_{T_{\ast}}:=\{\mathbf{\varpi}\in\rm{L}^2(0,T_{\ast};\mathbf{H}):\|\mathbf{\varpi}\|_{\rm{L}^2(0,T_{\ast};\mathbf{H})}\leq 1\}.$$
Observing that by the property (\ref{gin-xi-eta}), the diffusion term is monotonous as well the damping term, we can use the theory of monotone operators (\emph{cf.} \cite[Section 2.1]{Lions-1969} together with \cite[Section 9]{B-1989}) to prove that for each $\mathbf{\varpi}\in M_{T_{\ast}}$, there exists a weak solution $\mathbf{\upsilon}\in\mathrm{L}^{\infty}(0,T_{\ast};\mathbf{H})\cap\mathrm{L}^{q}(0,T_{\ast};\mathbf{V}_{q})\cap\mathbf{L}^{\sigma}(Q_{T_{\ast}})$ to the following system:
\begin{equation}\label{eq2-inc-ve}
\mathrm{div}\,\mathbf{\upsilon}=0\quad\mbox{in}\quad Q_{T_{\ast}},
\end{equation}
\begin{equation}\label{eq2-vel-ve}
\frac{\partial\,\mathbf{\upsilon}}{\partial\,t}-
\mathbf{div}(|\mathbf{\nabla\upsilon}|^{q-2}\mathbf{\nabla\upsilon})
+\alpha|\mathbf{\upsilon}|^{\sigma-2}\mathbf{\upsilon}
=\mathbf{f}-\mathbf{\nabla}p-\mathbf{div}(\mathbf{\varpi}\otimes\mathbf{\varpi}\Phi_{\epsilon}(|\mathbf{\varpi}|))
\quad\mbox{in}\quad Q_{T_{\ast}},
\end{equation}
\begin{equation}\label{eq1-ic-u-ve}
\mathbf{\upsilon}=\mathbf{u}_0\qquad\mbox{in}\quad\Omega \quad\mbox{for}\quad t=0,
\end{equation}
\begin{equation}\label{eq1-bc-u-ve}
\mathbf{\upsilon}=\mathbf{0}\qquad\mbox{on}\quad \Gamma_{T_{\ast}}.
\end{equation}
Moreover, once the diffusion and damping terms satisfy the monotonicity property (\ref{gin-xi-eta}), the weak solution of (\ref{eq2-inc-ve})-(\ref{eq1-bc-u-ve}) is unique. %

\vspace{0.2cm}
\noindent
\emph{Second Step.}
As a consequence of the previous step, we can define a mapping
\begin{equation}
\mathbf{K}:M_{T_{\ast}}\to\rm{L}^2(0,T_{\ast};\mathbf{H})
\end{equation}
such that to each $\mathbf{\varpi}\in M_{T_{\ast}}$ associates the unique weak solution
$\mathbf{\upsilon}\in\mathrm{L}^{\infty}(0,T_{\ast};\mathbf{H})\cap\mathrm{L}^{q}(0,T_{\ast};\mathbf{V}_{q})\cap\mathbf{L}^{\sigma}(Q_{T_{\ast}})$. %
Testing formally (\ref{eq2-vel-ve}) by the unique weak solution $\mathbf{\upsilon}:=\mathbf{K}(\mathbf{\varpi})$, with $\mathbf{\varpi}\in M_{T_{\ast}}$, integrating over $Q_t$, with $0<t<T_{\ast}$, using Young's inequality and, at last, the definition of $\Phi_{\epsilon}(|\mathbf{\varpi}|)$, we achieve to
\begin{equation}\label{est1-formaly}
\begin{split}
& \|\mathbf{\upsilon}\|_{\rm{L}^{\infty}(0,T_{\ast};\mathbf{H})}^2+
C_1\int_{Q_{T_{\ast}}}|\mathbf{\nabla}\mathbf{\upsilon}|^{q}d\mathbf{x}dt+
C_2\int_{Q_{T_{\ast}}}|\mathbf{\upsilon}|^{\sigma}d\mathbf{x}dt\leq \\
&  \gamma_1+\gamma_2\|\mathbf{\varpi}\|_{\rm{L}^{2}(0,T_{\ast};\mathbf{H})}^2,\qquad
\gamma_1:=\|\mathbf{u}_{0}\|_{\mathbf{H}}^2+C_3\int_{Q_{T_{\ast}}}|\mathbf{F}|^{q'}d\mathbf{x}ds,\quad
    \gamma_2:=C_4.
\end{split}
\end{equation}
Then setting $T_{\ast}:=\min\left\{1/(\gamma_1+\gamma_2),T\right\}$, we can prove, from (\ref{est1-formaly}) and due to the fact that $\mathbf{\varpi}\in M_{T_{\ast}}$, that
\begin{equation}\label{est-Kw-L2}
\|\mathbf{K}(\mathbf{\varpi})\|_{\rm{L}^{2}(0,T_{\ast};\mathbf{H})}^2\leq T_{\ast}(\gamma_1+\gamma_2)\leq 1
\end{equation}
for all $\mathbf{\varpi}\in M_{T_{\ast}}$. This proves that $\mathbf{K}$ maps $M_{T_{\ast}}$ into itself.

On the other hand, in order to prove the compactness of $\mathbf{K}$, we obtain from (\ref{est1-formaly}) that
\begin{equation}\label{est-Kw-Lg}
\|\mathbf{K}(\mathbf{\varpi})\|_{\rm{L}^{q}(0,T_{\ast};\mathbf{V}_{q})}\leq
C_1^{-1}(\gamma_1+\gamma_2)
\end{equation}
for all $\mathbf{\varpi}\in M_{T_{\ast}}$. %
Owing to the  assumptions (\ref{eq1-vel-u0}) and (\ref{h-f}), the right hand side of (\ref{est-Kw-Lg}) is finite.
Then, for the distributive time derivative $\mathbf{\upsilon}':=(\mathbf{K}(\mathbf{\varpi}))'$, with $\mathbf{\varpi}\in M_{T_{\ast}}$, we can prove that
\begin{equation*}
\|\mathbf{div}\left(|\mathbf{\nabla\upsilon}|^{q-2}\mathbf{\nabla\upsilon}
-\mathbf{\varpi}\otimes\mathbf{\varpi}\Phi_{\epsilon}(|\mathbf{\varpi}|)-\mathbf{F}\right)
\|_{\rm{L}^{q'}(0,T_{\ast};\mathbf{V}_{q}')}+
\||\mathbf{\upsilon}|^{\sigma-2}\mathbf{\upsilon}\|_{\mathbf{L}^{\sigma'}(Q_{T_{\ast}})}<\infty
\end{equation*}
and consequently
\begin{equation}\label{est1-v'-reg}
\mathbf{\upsilon}'\in\rm{L}^{q'}(0,T_{\ast};\mathbf{V}_{q}')+\mathbf{L}^{\sigma'}(Q_{T_{\ast}}).
\end{equation}
In fact, by virtue of (\ref{est1-formaly}), it follows the uniform boundedness of $|\mathbf{\nabla\upsilon}|^{q-2}\mathbf{\nabla\upsilon}$ in $\mathbf{L}^{q'}(Q_{T_{\ast}})$ and of $|\mathbf{\upsilon}|^{\sigma-2}\mathbf{\upsilon}$ in $\mathbf{L}^{\sigma'}(Q_{T_{\ast}})$. %
By assumption (\ref{h-f}), $\mathbf{F}\in\mathbf{L}^{q'}(Q_{T_{\ast}})$. %
On the other hand, using the definition of $\Phi_{\epsilon}$, we can prove that $\|\mathbf{\varpi}\otimes\mathbf{\varpi}\Phi_{\epsilon}(|\mathbf{\varpi}|)\|_{\mathbf{L}^{q'}(Q_{T_{\ast}})}^{q'}\leq C\|\mathbf{\varpi}\|_{\mathbf{L}^{2}(0,T_{\ast};\mathbf{H})}^2$. %
By (\ref{est-Kw-Lg}) and (\ref{est1-v'-reg}), and once that $\mathbf{V}_q\subset\subset\mathbf{H}\subset \mathbf{V}_q'$ for $q>\frac{2N}{N+2}$, we can apply Aubin-Lions compactness lemma (cf.~\cite{Simon-1987}) to prove that $\mathbf{K}(M_{T_{\ast}})$ is relatively compact in $\mathrm{L}^{q}(0,T_{\ast};\mathbf{H})$. %
Then, since $\mathbf{\upsilon}\in\rm{L}^{\infty}(0,T_{\ast};\mathbf{H})$, by parabolic interpolation it follows that
$\mathbf{K}(M_{T_{\ast}})$ is relatively compact in $\rm{L}^{2}(0,T_{\ast};\mathbf{H})$.

To prove the continuity of $\mathbf{K}$, we consider a sequence $\mathbf{\varpi}_m$ in $M_{T_{\ast}}$ such that
$$\mathbf{\varpi}_m\to \mathbf{\varpi}\quad\mbox{in}\quad \rm{L}^{2}(0,T_{\ast};\mathbf{H})\quad\mbox{as}\quad m\to\infty.$$
By the relative compactness of $\mathbf{K}(M_{T_{\ast}})$ in $\rm{L}^{2}(0,T_{\ast};\mathbf{H})$, there exists a subsequence $\mathbf{\varpi}_{m_k}$ such that
\begin{equation}\label{cont-subs-reg}
\mathbf{K}(\mathbf{\varpi}_{m_k})\to \mathbf{\upsilon}\quad\mbox{in}\quad \rm{L}^{2}(0,T_{\ast};\mathbf{H}),\quad\mbox{as}\quad k\to\infty.
\end{equation}
From the definition of $\mathbf{K}$, the functions $\mathbf{\upsilon}_{m_k}:=\mathbf{K}(\mathbf{\varpi}_{m_k})$ satisfy to
\begin{equation}\label{eq-ws-seq-reg}
\begin{split}
& -\int_{Q_{T_{\ast}}}\mathbf{\upsilon}_{m_k}\cdot\varphi_t\,d\mathbf{x}dt+
\int_{Q_{T_{\ast}}}|\mathbf{\nabla}\mathbf{\upsilon}_{m_k}|^{q-2}\mathbf{\nabla}\mathbf{\upsilon}_{m_k}:\mathbf{\nabla\varphi}\,d\mathbf{x}dt\\
&+
\alpha\int_{Q_{T_{\ast}}}|\mathbf{\upsilon}_{m_k}|^{\sigma-2}\mathbf{\upsilon}_{m_k}\cdot\mathbf{\varphi}\,d\mathbf{x}dt\\
& =\int_{Q_{T_{\ast}}}\left(\mathbf{F}+\mathbf{\varpi}_{m_k}\otimes\mathbf{\varpi}_{m_k}\Phi_{\epsilon}(|\mathbf{\varpi}_{m_k}|)\right):\mathbf{\nabla\varphi}\,d\,\mathbf{x}dt
+\int_{\Omega}\mathbf{u}_0\cdot\varphi(0)\,d\mathbf{x}
\end{split}
\end{equation}
for all $\varphi\in\mathbf{C}^{\infty}(Q_{T_{\ast}})$, with $\mathrm{div\,}\varphi=0$ and $\mathrm{supp\,}\varphi\subset\subset\Omega\times[0,T_{\ast})$. %
Passing to the limit in (\ref{eq-ws-seq-reg}) by using the appropriated convergence results (see \cite[p. 119]{Wolf-2007} and \cite[p. 236]{B-1989}) and the usual Minty trick (see \emph{e.g.}~\cite[pp. 212-214]{Lions-1969}), we can prove that $\mathbf{\upsilon}=\mathbf{K}(\mathbf{\varpi})$. %
The only difference here is that $|\mathbf{\upsilon}_{m_k}|^{\sigma-2}\mathbf{\upsilon}_{m_k}\to \widetilde{\mathbf{\upsilon}}$ weakly in $\mathbf{L}^{\sigma'}(Q_{T_{\ast}})$, as $k\to\infty$. %
Since $\mathbf{\upsilon}_{m_k}\to\mathbf{\upsilon}$ weakly in $\mathbf{L}^{\sigma}(Q_{T_{\ast}})$, as $k\to\infty$, there exists a subsequence, still denoted by $\mathbf{\upsilon}_{m_k}$, such that $\mathbf{\upsilon}_{m_k}\to\mathbf{\upsilon}$ a.e. in $Q_{T_{\ast}}$.
In addition, because $|\mathbf{\upsilon}_{m_k}|^{\sigma-2}\mathbf{\upsilon}_{m_k}$ is uniformly bounded in $\mathbf{L}^{\sigma'}(Q_{T_{\ast}})$, we can apply Lesbesgue's theorem of dominated convergence to prove that
\begin{equation}\label{str-conv-yup-sig'}
 |\mathbf{\upsilon}_{m_k}|^{\sigma-2}\mathbf{\upsilon}_{m_k}\to|\mathbf{\upsilon}|^{\sigma-2}\mathbf{\upsilon}\quad
 \mbox{strongly in}\ \mathbf{L}^{\sigma'}(Q_{T_{\ast}})
\end{equation}
and, as a consequence, $\widetilde{\mathbf{\upsilon}}=|\mathbf{\upsilon}|^{\sigma-2}\mathbf{\upsilon}$.
From (\ref{cont-subs-reg}), we conclude that
$\mathbf{\mathbf{K}}(\mathbf{\varpi}_m)\to \mathbf{K}(\mathbf{\varpi})$ in $\rm{L}^{2}(0,T_{\ast};\mathbf{H})$ as $m\to\infty$, which proves the continuity of $\mathbf{K}$.

Now, applying Schauder's fixed point theorem, there exists a function $\mathbf{\upsilon}_{T_{\ast}}\in M_{T_{\ast}}$  such that $\mathbf{K}(\mathbf{\upsilon}_{T_{\ast}})=\mathbf{\upsilon}_{T_{\ast}}$ and which is a weak solution to the problem (\ref{eq2-inc-e})-(\ref{eq1-bc-u-e}) in the cylinder $Q_{T_{\ast}}$.

\vspace{0.2cm}
\noindent
\emph{Third Step.}
Testing (\ref{eq2-vel-e}) by the weak solution $\mathbf{\upsilon}_{T_{\ast}}$, integrating over $Q_{T_{\ast}}$, proceeding we did as for (\ref{est1-formaly}) and observing that due to the definition of $\Phi_{\varepsilon}$ the term resulting from convection is zero, we obtain
\begin{equation}\label{est2-formaly}
\|\mathbf{\upsilon}_{T_{\ast}}(T_{\ast})\|_{\mathbf{H}}^2\leq
C\left(\|\mathbf{u}_{0}\|_{\mathbf{H}}^2+
\|\mathbf{F}\|_{\mathbf{L}^{q'}(Q_{T_{\ast}})}^{q'}\right).
\end{equation}
The estimate (\ref{est2-formaly}) is independent of $T_{\ast}$ and therefore we can extend $\mathbf{\upsilon}_{T_{\ast}}$ as a weak solution to the problem (\ref{eq2-inc-e})-(\ref{eq1-bc-u-e}) in the whole cylinder $Q_T$.

Finally, the energy relation (\ref{e-equality-qr}) follows by testing (\ref{eq2-vel-e}) by a weak solution and integrating over $Q_t$ with $0<t<T_{\ast}$.
\end{proof}

\section{Existence of approximate solutions}\label{Sect-EAS}

Let $\mathbf{u}_{\epsilon}\in\rm{L}^{q}(0,T;\mathbf{V}_{q})\cap\mathrm{L}^{\infty}(0,T;\mathbf{H})\cap\mathbf{L}^{\sigma}(Q_T)$ be a weak solution to the problem (\ref{eq2-inc-e})-(\ref{eq1-bc-u-e}). %
From Proposition~\ref{th-uxu-reg} (see (\ref{e-equality-qr})), we can prove that
\begin{equation}\label{e-inequality-qr}
 \|\mathbf{u}_{\epsilon}\|_{\rm{L}^{\infty}(0,T;\mathbf{H})}^2+
\int_{Q_T}|\mathbf{\nabla}\mathbf{u}_{\epsilon}|^{q}d\mathbf{x}dt+
\int_{Q_T}|\mathbf{u}_{\epsilon}|^{\sigma}d\mathbf{x}dt\leq C,
\end{equation}
where, by the assumptions (\ref{eq1-vel-u0}) and (\ref{h-f}), $C$ is a positive constant which does not depend on $\epsilon$. %
From (\ref{e-inequality-qr}) we obtain
\begin{equation}\label{est-inf-gam}
\|\mathbf{u}_{\epsilon}\|_{L^{\infty}(0,T;\mathbf{H})}^2+\|\mathbf{u}_{\epsilon}\|_{\rm{L}^{q}(0,T;\mathbf{V}_{q})}^{q}\leq C,
\end{equation}
\begin{equation}\label{est-sig}
\|\mathbf{u}_{\epsilon}\|_{\mathbf{L}^{\sigma}(Q_T)}\leq C.
\end{equation}
Using (\ref{est-inf-gam}) and (\ref{est-sig}), it follows that
\begin{equation}\label{est-q'}
\||\mathbf{\nabla}\mathbf{u}_{\epsilon}|^{q-2}\mathbf{\nabla}\mathbf{u}_{\epsilon}\|_{\mathbf{L}^{q'}(Q_T)}\leq C,
\end{equation}
\begin{equation}\label{est-sig'}
\||\mathbf{u}_{\epsilon}|^{\sigma-2}\mathbf{u}_{\epsilon}\|_{\mathbf{L}^{\sigma'}(Q_T)}\leq C.
\end{equation}
On the other hand, by using (\ref{est-inf-gam}) and the
Sobolev imbedding $\mathrm{L}^q(0,T;\mathbf{V}_q)\cap\mathrm{L}^{\infty}(0,T;\mathbf{H})\hookrightarrow\mathrm{L}^{q\frac{N+2}{N}}(Q_T)$ (see~\cite[p. 213]{Lions-1969}), we can prove that
\begin{equation}\label{est-gam(N+2)/N}
\|\mathbf{u}_{\epsilon}\|_{\rm{L}^{q\frac{N+2}{N}}(Q_T)}\leq C.
\end{equation}
As a consequence of (\ref{est-gam(N+2)/N}) and of the definition of $\Phi_{\epsilon}$ (see (\ref{Phi-e})),
\begin{equation}\label{est-Phiuxu}
\|\mathbf{u}_{\epsilon}\otimes\mathbf{u}_{\epsilon}\Phi_{\epsilon}(|\mathbf{u}_{\epsilon}|)\|_{\mathbf{L}^{q\frac{N+2}{2N}}(Q_T)}\leq C.
\end{equation}
Note that the constants in (\ref{est-inf-gam})-(\ref{est-Phiuxu}) are distinct and do not depend on $\epsilon$. %
From (\ref{est-inf-gam})-(\ref{est-Phiuxu}), there exists a sequence of positive numbers $\epsilon_m$ such that
$\epsilon_m\to 0$, as $m\to \infty$, and
\begin{equation}\label{convg-La}
\mbox{$\mathbf{u}_{\epsilon_m}\to \mathbf{u}$\quad weakly in $\rm{L}^{q}(0,T;\mathbf{V}_{q})$,\quad as $m\to\infty$,}
\end{equation}
\begin{equation}\label{convg-Lsig}
\mbox{$\mathbf{u}_{\epsilon_m}\to \mathbf{u}$\quad weakly in $\mathbf{L}^{\sigma}(Q_T)$,\quad as $m\to\infty$,}
\end{equation}
\begin{equation}\label{convg-S-b'}
\mbox{$|\mathbf{\nabla}\mathbf{u}_{\epsilon_m}|^{q-2}\mathbf{\nabla}\mathbf{u}_{\epsilon_m}\to \mathbf{S}$\quad weakly in $\mathbf{L}^{q'}(Q_T)$,\quad as $m\to\infty$,}
\end{equation}
\begin{equation}\label{convg-sig-b'}
\mbox{$|\mathbf{u}_{\epsilon_m}|^{\sigma-2}\mathbf{u}_{\epsilon_m}\to \widetilde{\mathbf{u}}$\quad weakly in $\mathbf{L}^{\sigma'}(Q_T)$,\quad as $m\to\infty$,}
\end{equation}
\begin{equation}\label{convg-La(N+2)/N}
\mbox{$\mathbf{u}_{\epsilon_m}\to \mathbf{u}$\quad weakly in $\rm{L}^{q\frac{N+2}{N}}(0,T;\mathbf{V}_{q})$,\quad as $m\to\infty$,}
\end{equation}
\begin{equation}\label{convg-Phiuxu}
\mbox{$\mathbf{u}_{\epsilon_m}\otimes\mathbf{u}_{\epsilon_m}\Phi_{\epsilon_m}(|\mathbf{u}_{\epsilon_m}|)\to \mathbf{G}$\quad weakly in $\mathbf{L}^{q\frac{N+2}{2N}}(Q_T)$,\quad as $m\to\infty$}.
\end{equation}
Here we observe that using (\ref{convg-sig-b'}) and arguing as in the proof of Proposition~\ref{th-uxu-reg} (see (\ref{str-conv-yup-sig'})), we can prove that
\begin{equation}\label{str-conv-u-sig'}
 |\mathbf{u}_{\epsilon_m}|^{\sigma-2}\mathbf{u}_{\epsilon_m}\to|\mathbf{u}|^{\sigma-2}\mathbf{u}\quad
 \mbox{strongly in}\ \mathbf{L}^{\sigma'}(Q_{T}).
\end{equation}
As a consequence, we ca write
$\widetilde{\mathbf{u}}=|\mathbf{u}|^{\sigma-2}\mathbf{u}$. %
Then, using the convergence results (\ref{convg-La})-(\ref{str-conv-u-sig'}), we can pass to the limit $\epsilon_m\to 0$ in (\ref{eq-ws-reg}) with $\mathbf{u}_{\epsilon}$ replaced by $\mathbf{u}_{\epsilon_m}$, to obtain
\begin{equation}\label{limit-eq-SH}
-\int_{Q_T}\mathbf{u}\cdot\varphi_t\,d\mathbf{x}dt+\alpha\int_{Q_T}|\mathbf{u}|^{\sigma-2}\mathbf{u}\cdot\mathbf{\varphi}\,d\mathbf{x}dt+
\int_{Q_T}(\mathbf{S}-\mathbf{G}-\mathbf{F}):\mathbf{\nabla\varphi}\,d\mathbf{x}dt
=\int_{\Omega}\mathbf{u}_0\cdot\varphi(0)\,d\mathbf{x}
\end{equation}
for all $\varphi\in\mathbf{C}^{\infty}(Q_T)$, with $\mathrm{div\,}\varphi=0$ and $\mathrm{supp\,}\varphi\subset\subset\Omega\times[0,T)$. %

\section{Convergence of the approximated convective term}\label{Sect-CACT}

In this section we shall prove that $\mathbf{G}=\mathbf{u}\otimes\mathbf{u}$. %
We start by observing that, from (\ref{eq-ws-reg}), it follows
\begin{equation}\label{eq-ws-reg-em}
\begin{split}
& -\int_{Q_T}\mathbf{u}_{\epsilon_m}\cdot\varphi_t\,d\mathbf{x}dt+
   \alpha\int_{Q_T}|\mathbf{u}_{\epsilon_m}|^{\sigma-2}\mathbf{u}_{\epsilon_m}\cdot\varphi\,d\mathbf{x}dt+\\
& \int_{Q_T}\left(|\mathbf{\nabla}\mathbf{u}_{\epsilon_m}|^{q-2}\mathbf{\nabla}\mathbf{u}_{\epsilon_m}-\mathbf{u}_{\epsilon_m}\otimes\mathbf{u}_{\epsilon_m}\Phi_{\epsilon_m}(|\mathbf{u}_{\epsilon_m}|)-\mathbf{F}\right):\mathbf{\nabla\varphi}\,d\mathbf{x}dt=0
\end{split}
\end{equation}
for all $\varphi\in\mathbf{C}^{\infty}_0(Q_T)$ with $\mathrm{div\,}\varphi=0$. %
Then, from (\ref{h-f}), (\ref{est-q'}) and (\ref{est-Phiuxu}), we have
\begin{equation}\label{THF-Lr}
\mathbf{Q}_{\epsilon_m}:=|\mathbf{\nabla}\mathbf{u}_{\epsilon_m}|^{q-2}\mathbf{\nabla}\mathbf{u}_{\epsilon_m}-\mathbf{u}_{\epsilon_m}\otimes\mathbf{u}_{\epsilon_m}\Phi_{\epsilon_m}(|\mathbf{u}_{\epsilon_m}|)-\mathbf{F}
\in
\mathbf{L}^r(Q_T) %
\end{equation}
for any $r$ satisfying to
\begin{equation}\label{r}
1<r\leq\min\left\{\frac{q(N+2)}{2N},q'\right\}.
\end{equation}
Using (\ref{est-sig'}) and (\ref{THF-Lr})-(\ref{r}), we can obtain, from (\ref{eq-ws-reg-em}), that
the distributive time derivatives
\begin{equation}\label{xi-ve'}
\mathbf{u}_{\epsilon_m}'\in \mathrm{L}^{r}(0,T;\mathbf{W}^{-1,r}(\Omega))+\mathbf{L}^{\sigma'}(Q_T).
\end{equation}
Due to the admissible range for $r$ (see (\ref{r})), there always exists a $\gamma>1$ such that the following compact and Sobolev imbeddings hold
\begin{equation}\label{c-imb}
\mathbf{W}^{1,q}_0(\Omega)\subset\subset\mathbf{L}^{\gamma}(\Omega)\subset\mathbf{W}^{-1,r}(\Omega),\quad
((r')^{\ast})'\leq\gamma<q^{\ast},
\end{equation}
where $q^{\ast}$ is the Sobolev conjugate of $q$ and $r'$ is the Hölder conjugate of $r$. %
Then, using Aubin-Lions compactness lemma (cf. Simon~\cite{Simon-1987}), we obtain from (\ref{convg-La}) together with (\ref{xi-ve'}) and (\ref{c-imb}), and passing to a subsequence, that
\begin{equation}\label{comp-g}
\mathbf{u}_{\epsilon_{m}}\to\mathbf{u}\quad
\mbox{strongly in}\quad \mathrm{L}^{r}(0,T;\mathbf{L}^{\gamma}(\Omega)),\quad\mbox{as}\quad m\to \infty.
\end{equation}
Using parabolic interpolation, we obtain from (\ref{est-inf-gam}) and (\ref{comp-g}) that
\begin{equation}\label{str-conv-g}
\mathbf{u}_{\epsilon_{m}}\to\mathbf{u}\quad
\mbox{strongly in}\quad \mathrm{L}^{s}(0,T;\mathbf{L}^{\gamma}(\Omega))\quad\forall\ s:1\leq s<\infty,\quad\mbox{as}\quad m\to \infty.
\end{equation}
Now, observing that $q>\frac{2N}{N+2}$ is equivalent to $q^{\ast}>q\frac{N+2}{N}$, we can choose $\gamma$ such that $q\frac{N+2}{N}\leq\gamma<q^{\ast}$ such that, in view of (\ref{str-conv-g}),
\begin{equation*}\label{str-conv-cve}
\mathbf{u}_{\epsilon_{m}}\to\mathbf{u}\quad
\mbox{strongly in}\quad \mathbf{L}^{s}(Q_T),\quad s=q\frac{N+2}{N},\quad\mbox{as}\quad m\to \infty.
\end{equation*}
In consequence
\begin{equation}\label{unb-uxu-em}
\mathbf{u}_{\epsilon_m}\otimes\mathbf{u}_{\epsilon_m}\Phi_{\epsilon_m}(|\mathbf{u}_{\epsilon_m}|)\to                                    \mathbf{u}\otimes\mathbf{u}\quad\mbox{strongly in $\mathbf{L}^{q\frac{N+2}{2N}}(Q_T)$},\quad \mbox{as}\ m\to\infty.
\end{equation}
Finally, from (\ref{convg-Phiuxu}) and (\ref{unb-uxu-em}), we conclude that $\mathbf{G}=\mathbf{u}\otimes\mathbf{u}$.

\section{Weak continuity}\label{Sect-WC}

We start this section by proving that
\begin{equation}\label{u-wcont}
\mathbf{u}\in\mathrm{C}_w([0,T];\mathbf{H}).
\end{equation}
We observe that, from (\ref{limit-eq-SH}), the distributive time derivative $\mathbf{u}_t$ is uniquely defined by
\begin{equation}\label{dist-ti-der}
\langle\mathbf{u}_t,\mathbf{\varphi}\rangle=\langle\mathbf{div}(\mathbf{S}-\mathbf{G})-\alpha|\mathbf{u}|^{\sigma-2}\mathbf{u}+\mathbf{F},\mathbf{\varphi}\rangle\quad \forall\ \mathbf{\varphi}\in\mathrm{C}^{\infty}_0(0,T;\mathbf{Y}),
\end{equation}
where
\begin{equation}\label{set-Y}
\mathbf{Y}:=\mathbf{V}_{q}\cap\mathbf{V}_{\kappa}\cap\mathbf{L}^{\sigma}(Q_T)\cap\mathbf{H},\quad
\kappa:=\left(\frac{q(N+2)}{2N}\right)'.
\end{equation}
Then we can prove that
\begin{equation}\label{ut-in-Lr}
\mathbf{u}_t\in\mathrm{L}^{\frac{N+1}{N}}(0,T;\mathbf{Y}').
\end{equation}
In fact, due to (\ref{convg-S-b'}) and (\ref{convg-sig-b'}), immediately follows that $\mathbf{div}\,\mathbf{S}\in\mathrm{L}^{\frac{N+1}{N}}(0,T;\mathbf{V}_{q}')$ and $|\mathbf{u}|^{\sigma-2}\mathbf{u}\in\mathbf{L}^{\sigma'}(Q_T)$, respectively.
By assumption (\ref{h-f}), $\mathbf{F}\in\mathrm{L}^{\frac{N+1}{N}}(0,T;\mathbf{V}_{q}')$. %
That $\mathbf{div}\,\mathbf{G}\in\mathrm{L}^{\frac{N+1}{N}}(0,T;\mathbf{V}_{\kappa}')$ follows by (\ref{convg-Phiuxu}) if $k'=q\frac{N+2}{2N}$, which in fact is true by our choice of $k$ (\emph{cf.}~(\ref{set-Y})). %

Next, let $t_0\in[0,T]$ be fixed and let $t_k$ be a sequence in $[0,T]$ such that
\begin{equation*}
t_k\to t_0,\quad \mbox{as}\ k\to\infty,\quad\mbox{and such that}\quad \mathbf{u}(t_k)\in\mathbf{H}\quad \forall\ k\in\mathds{N}.
\end{equation*}
Then we consider the continuous representant of $\mathbf{u}$ in $\mathrm{C}(0,T;\mathbf{Y}')$, which exists by virtue of (\ref{dist-ti-der}) and (\ref{ut-in-Lr}). %
Finally by means of reflexivity in $\mathbf{H}$ and of the continuous and dense imbedding of $\mathbf{H}$ into $\mathbf{Y}'$, we can prove that
\begin{equation*}
\mathbf{u}(t_k)\to \mathbf{u}(t_0)\quad \mbox{weakly in}\ \mathbf{H},\quad \mbox{as}\ k\to\infty,
\end{equation*}
and whence (\ref{u-wcont}).

Now, let us prove that for every $t\in[0,T]$
\begin{equation}\label{w-cont-em}
\mathbf{u}_{\epsilon_m}(t)\to\mathbf{u}(t)\quad\mbox{weakly in}\ \mathbf{H},\quad \mbox{as}\ m\to\infty.
\end{equation}
Due to (\ref{est-inf-gam}), there exists a subsequence $\mathbf{u}_{\epsilon_{m_k}}(t)\in\mathbf{H}$ such that
\begin{equation*}
\mathbf{u}_{\epsilon_{m_k}}(t)\to\eta\quad\mbox{weakly in}\ \mathbf{H},\quad \mbox{as}\ m\to\infty.
\end{equation*}
Arguing as we did for (\ref{dist-ti-der}) and (\ref{ut-in-Lr}), the distributive time derivative $\mathbf{u}_{\epsilon_{m_k}}'\in\mathrm{L}^{\frac{N+1}{N}}(0,T;\mathbf{Y}')$ and is uniquely defined by
\begin{equation}\label{dist-t-der-em}
\begin{split}
  & \langle\mathbf{u}_{\epsilon_{m_k}}',\mathbf{\varphi}\rangle= \\
  & \langle\mathbf{div}(|\mathbf{\nabla}\mathbf{u}_{\epsilon_{m_k}}|^{q-2}\mathbf{\nabla}\mathbf{u}_{\epsilon_{m_k}}-
\mathbf{u}_{\epsilon_{m_k}}\otimes\mathbf{u}_{\epsilon_{m_k}}\Phi_{\epsilon_{m_k}}(|\mathbf{u}_{\epsilon_{m_k}}|))-
\alpha|\mathbf{u}_{\epsilon_{m_k}}|^{\sigma-2}\mathbf{u}_{\epsilon_{m_k}}+\mathbf{F},\mathbf{\varphi}\rangle
\end{split}
\end{equation}
for all $\mathbf{\varphi}\in\mathrm{C}^{\infty}_0(0,T;\mathbf{Y})$. %
In particular, there holds $\mathbf{u}_{\epsilon_{m_k}}\in\mathrm{C}(0,T;\mathbf{Y}')$. %
Next, we introduce $\eta$ in (\ref{dist-t-der-em}), we use integration by parts and we carry out the passage to the limit in the resulting equation
by using the convergence results (\ref{convg-La})-(\ref{convg-Phiuxu}).
Combining this equation with the one which results from inserting $\eta$ into (\ref{dist-ti-der}) and integrating by parts, we obtain
$\mathbf{u}(t)=\eta$, which yields (\ref{w-cont-em}). %
Finally, combining (\ref{u-wcont}) and (\ref{w-cont-em}), we see that also $\mathbf{u}_{\epsilon_{m}}$ satisfies to (\ref{u-wcont}).

\section{Auxiliary results for decomposing the pressure}\label{Sect-ARDP}

Here we make a break in the proof of Theorem~\ref{th-exst-wsqv} to show that the results of Wolf~\cite{DRW-2010} concerned with the local decomposition of the pressure still hold in the case of the momentum equation modified by the presence of the damping term. %
For, let $\omega'$ be a fixed but arbitrary open bounded subset of $\Omega$ such that
\begin{equation}\label{om-l}
\omega'\subset\subset\Omega\quad\mbox{and}\quad \partial\omega'\in\mathrm{C}^2. %
\end{equation}
Given $s$ such that $1<s<\infty$, lets us consider the following auxiliary function spaces related with the Helmholtz-Weyl decomposition (\emph{cf.} \cite[Section 2]{Wolf-2007}, see also \cite[Section III.1]{Galdi-1994} and \cite{SS-1992}):
\begin{eqnarray}
  & &\label{set-A}
\mathrm{A}^s(\omega'):=\{a\in\mathrm{L}^s(\omega'):a=\triangle f,\ f\in\mathrm{W}^{2,s}_0(\omega')\}; \\
  & &
\mathrm{\dot{B}}^s(\omega'):=\left\{b\in\mathrm{B}^s(\omega'):\int_{\omega'} b\,d\mathbf{x}=0\right\},\quad
\mathrm{B}^s(\omega'):=\left\{b\in\mathrm{L}^{s}(\omega'):\triangle b=0\right\}.\label{set-Bpt}
\end{eqnarray}

\begin{proposition}\label{prop1-dec-p}
Let $\mathbf{Q}\in\mathbf{L}^{s_1}(\omega'_T)$, $\mathbf{q}\in\mathbf{L}^{s_2}(\omega'_T)$, with $1<s_1,\ s_2<\infty$,
and
\begin{equation}\label{u-wc-P1}
\mathbf{u}\in\mathrm{C}_{\rm{w}}([0,T];\mathbf{H})
\end{equation}
where here $\mathbf{H}$ is defined over $\omega'$. %
Suppose that
\begin{equation}\label{eq1-pro-dp}
-\int_{\omega'_T}\mathbf{u}\cdot\varphi_t\,d\mathbf{x}dt+
 \int_{\omega'_T}\mathbf{Q}:\mathbf{\nabla\varphi}\,d\mathbf{x}dt+
 \int_{\omega'_T}\mathbf{q}\cdot\varphi\,d\mathbf{x}dt=0
\end{equation}
for all $\varphi\in\mathbf{C}^{\infty}_0(\omega'_T)$ with $\mathrm{div\,}\varphi=0$ and where $\omega'_T=\omega'\times(0,T)$ and $\omega'$ satisfies to (\ref{om-l}). %
Then there exist unique functions
\begin{equation}\label{eq2-pro-dp}
p^0\in\mathrm{L}^{s_0}(0,T;A^{s_0}(\omega')),
\end{equation}
\begin{equation}\label{eq3-pro-dp}
\tilde{p}^h\in\mathrm{C}_{\rm{w}}([0,T];\dot{B}^{s_0}(\omega')),
\end{equation}
where $s_0$ can be taken such that
\begin{equation}\label{r0}
1<s_0\leq\min\left\{s_1,s_2\right\},
\end{equation}
such that
\begin{equation}\label{eq-ws-reg-f-ath}
\begin{split}
&
-\int_{\omega'_T}\mathbf{u}\cdot\varphi_t\,d\mathbf{x}dt
+\int_{\omega'_T}\mathbf{Q}:\mathbf{\nabla}\mathbf{\varphi}\,d\mathbf{x}dt
+\int_{\omega'_T}\mathbf{q}\cdot\mathbf{\varphi}\,d\mathbf{x}dt\\
&
=
\int_{\omega'_T}p^0\mathrm{div}\mathbf{\varphi}\,d\mathbf{x}dt-
\int_{\omega'_T}\tilde{p}^h\frac{\partial\mathrm{div}\mathbf{\varphi}}{\partial\,t}\,d\mathbf{x}dt+
\int_{\omega'}\mathbf{u}(0)\cdot\mathbf{\varphi}(0)\,dt
\end{split}
\end{equation}
for all $\varphi\in\mathbf{C}^{\infty}(\omega'_T)$, with $\mathrm{supp\,}\varphi\subset\subset\omega'\times[0,T)$. %
In addition, the following estimates hold
\begin{equation}\label{b-p0-f}
\|p^0\|_{\mathbf{L}^{s_0}(\omega'_T)}\leq
C_1\left(\|\mathbf{Q}\|_{\mathbf{L}^{s_1}(\omega'_T)}+\|\mathbf{q}\|_{\mathbf{L}^{s_2}(\omega'_T)}\right),
\end{equation}
\begin{equation}\label{b-ph-f}
\|\tilde{p}^h\|_{\mathrm{L}^{\infty}(0,T;\mathbf{L}^{s_0}(\omega'))}
\leq C_2 \left(
\|\mathbf{u}\|_{\mathrm{L}^{\infty}(0,T;\mathbf{L}^2(\omega'))}+
\|\mathbf{Q}\|_{\mathbf{L}^{s_1}(\omega'_T)}+
\|\mathbf{q}\|_{\mathbf{L}^{s_2}(\omega'_T)}\right).
\end{equation}
where $C_1$ and $C_2$ are positive constants depending only on $s_i$ ($i=0,1,2$), $N$ and $\omega'_T$.
\end{proposition}

\begin{proof}
Let $\psi\in\mathbf{C}^{\infty}_0(\omega')$ with $\mathrm{div}\psi=0$ and let $\eta\in\mathrm{C}^{\infty}_0(0,T)$. %
Inserting $\mathbf{\varphi}=\mathbf{\psi}\eta$ into (\ref{eq1-pro-dp}) and using Fubini's theorem, we obtain
\begin{equation*}
-\int_0^T\alpha\eta'\,dt=\int_0^T\beta\eta\,dt+\int_0^T\gamma\eta\,dt,
\end{equation*}
where for $t\in [0,T]$
\begin{equation*}
\alpha(t):=\int_{\omega'}\mathbf{u}(t)\cdot\mathbf{\psi}\,d\mathbf{x},\quad
\beta(t):=\int_{\omega'}\mathbf{Q}(t):\mathbf{\nabla\,\psi}\,d\mathbf{x},\quad
\gamma(t):=\int_{\omega'}\mathbf{q}(t)\cdot\mathbf{\psi}\,d\mathbf{x}.
\end{equation*}
Since $\mathbf{Q}\in\mathbf{L}^{s_1}(\omega'_T)$ and  $\mathbf{q}\in\mathbf{L}^{s_2}(\omega'_T)$, we have $\beta\in\mathrm{L}^{s_1}(0,T)$ and $\gamma\in\mathrm{L}^{s_2}(0,T)$. In consequence, $\alpha\in\mathrm{W}^{1,s_0}(\omega'_T)$ for any $s_0$ such that $1<s_0\leq\min\{s_1,s_2\}$. %
By Sobolev's imbedding theorem, $\alpha$ is represented by a continuous function, which we still denote by $\alpha$. %
Using integration by parts, we can represent
\begin{equation}\label{abg-t}
\alpha(t)=\alpha(0)+\int_0^t\beta(s)\,ds+\int_0^t\gamma(s)\,ds\quad \forall\ t\in(0,T).
\end{equation}
Let $t\in(0,T)$ be arbitrarily chosen. %
Using Fubini's theorem, the identity (\ref{abg-t}) reads
\begin{equation*}\label{eq01-p-Bog}
\int_{\omega'}\left[\left(\mathbf{u}(t)-\mathbf{u}(0)\right)\cdot\mathbf{\psi}+\widetilde{\mathbf{q}}(t)\cdot\mathbf{\psi}
+\widetilde{\mathbf{Q}}(t):\mathbf{\nabla\psi}\right]d\mathbf{x}=0,
\end{equation*}
where
\begin{equation*}
\widetilde{\mathbf{Q}}(t):=\int_0^t\mathbf{Q}(s)\,ds,\quad
\widetilde{\mathbf{q}}(t):=\int_0^t\mathbf{q}(s)\,ds.
\end{equation*}
Now, by the results of M.E.~Bogovski\u{i} (\emph{cf.} \cite[Theorem III.3.1]{Galdi-1994}) and of K.I.~Piletskas (\emph{cf.} \cite[Theorem III.5.2]{Galdi-1994}), there exists a unique function
\begin{equation*}
\tilde{p}(t)\in\mathrm{L}^{s_0}(\omega')\quad\mbox{with}\ \int_{\omega'}\tilde{p}(t)\,d\mathbf{x}=0
\end{equation*}
such that
\begin{equation}\label{eq1-p-Bog}
\int_{\omega'}\left[\left(\mathbf{u}(t)-\mathbf{u}(0)\right)\cdot\mathbf{\psi}+\widetilde{\mathbf{q}}(t)\cdot\mathbf{\psi}
+\widetilde{\mathbf{Q}}(t):\mathbf{\nabla\psi}\right]d\mathbf{x}=
\int_{\omega'}\tilde{p}(t)\mathrm{div}\psi\,d\mathbf{x}
\end{equation}
for all $\psi\in\mathbf{W}^{1,s_0'}_0(\omega')$. %
In addition,
\begin{equation}\label{pro-b-ph-f}
\|\tilde{p}^h(t)\|_{\mathbf{L}^{s_0}(\omega')}
\leq C\left(
\|\mathbf{u}(t)-\mathbf{u}(0)\|_{\mathbf{L}^{s_0}(\omega')}+
\|\widetilde{\mathbf{Q}}(t)\|_{\mathbf{L}^{s_0}(\omega')}+
\|\widetilde{\mathbf{q}}(t)\|_{\mathbf{L}^{s_0}(\omega')}\right).
\end{equation}
On the other hand, by the application of Helmholtz-Weyl decomposition of $\mathbf{L}^{s_0}(\omega')$ (\emph{cf.} \cite[Theorem 1.4]{SS-1992}), there exist $\tilde{p}^0(t)\in A^{s_0}(\omega')$ and $\tilde{p}^h(t)\in\dot{B}^{s_0}(\omega')$ such that
\begin{equation}\label{ptil-0h}
\tilde{p}(t)=\tilde{p}^0(t)+\tilde{p}^h(t)\quad\mbox{in}\quad \omega'
\end{equation}
and where $\mathrm{A}^r(\omega')$ and $\mathrm{\dot{B}}^r(\omega')$ are defined in  (\ref{set-A})-(\ref{set-Bpt}).
Moreover, the sum $A^{s_0}(\omega')+\dot{B}^{s_0}(\omega')$ is direct. %
Now, from (\ref{eq1-p-Bog}) and, as a consequence of the assumption (\ref{u-wc-P1}), we can infer that
\begin{equation}\label{ptil-Cw}
\tilde{p}\in\mathrm{C}_{\rm{w}}([0,T];\mathrm{L}^{s_0}(\omega')).
\end{equation}
From (\ref{ptil-0h}) and (\ref{ptil-Cw}) it follows that
\begin{equation}\label{ptil-0-Cw}
\tilde{p}^0\in\mathrm{C}_{\rm{w}}([0,T];A^{s_0}(\omega')).
\end{equation}
\begin{equation}\label{ptil-h-Cw}
\tilde{p}^h\in\mathrm{C}_{\rm{w}}([0,T];\dot{B}^{s_0}(\omega')).
\end{equation}
As a consequence of (\ref{pro-b-ph-f}), (\ref{ptil-0h}), we can derive (\ref{b-ph-f}). %
Moreover, inserting $\psi=\nabla\phi$ in (\ref{eq1-p-Bog}), for $\phi\in\mathrm{C}_0^{\infty}(\omega')$, using (\ref{ptil-0h}), integrating by parts the resulting equation and observing that, in view of (\ref{ptil-h-Cw}), $\triangle\tilde{p}^h=0$, and, in of view of (\ref{u-wc-P1}), $\mathrm{div}\mathbf{u}=0$, we obtain
\begin{equation}\label{eq2-p-Bog}
\int_{\omega'}\left(\widetilde{\mathbf{q}}(t)\cdot\mathbf{\nabla}\phi
+\widetilde{\mathbf{Q}}(t):\mathbf{\nabla}^2\phi\right)d\mathbf{x}=
\int_{\omega'}\tilde{p}^0(t)\triangle\phi\,d\mathbf{x}
\end{equation}
for all $\phi\in\mathrm{C}_0^{\infty}(\omega')$. %
Now, using (\ref{eq2-p-Bog}) and proceeding as in \cite[pp. 115-116]{Wolf-2007}, we can prove that $\tilde{p}^0\in\mathrm{W}^{1,s_0}(0,T;A^{s_0}(\omega'))$ and
\begin{equation}\label{pro-b1-ph-f}
\|p^0\|_{\mathbf{L}^{s_0}(\omega'_T)}
\leq C\left(
\|\mathbf{Q}\|_{\mathbf{L}^{s_1}(\omega'_T)}+
\|\mathbf{q}\|_{\mathbf{L}^{s_2}(\omega'_T)}\right),\quad\mbox{where}\quad p^0:=\frac{\partial\,\tilde{p}^0}{\partial\,t}
\end{equation}
and the constant $C$ depends only on $s_i$ ($i=1,2,3$), $N$ and $\omega'_T$. %
Whence (\ref{b-p0-f}) holds.
Finally, the identity (\ref{eq-ws-reg-f-ath}) follows by integrating (\ref{eq1-p-Bog}) over $(0,T)$, replacing there $\varphi$ by $\frac{\partial\,\varphi}{\partial\,t}$ and using (\ref{ptil-0h}), (\ref{ptil-0-Cw}), (\ref{ptil-h-Cw}) and the definition of $p^0$ given in (\ref{pro-b1-ph-f}).
The uniqueness of $p^0$ and $\tilde{p}^h$ follow from (\ref{b-p0-f}) and (\ref{b-ph-f}), respectively.
\end{proof}

\section{Decomposition of the pressure.}\label{Sect-Dec-p}

Let us continue with the proof of Theorem~\ref{th-exst-wsqv}. Using the results of the previous section, we shall decompose the pressure into two different components. %
For, let $\omega'$ be a domain in the conditions of the previous section (see (\ref{om-l})).
Clearly, in view of (\ref{eq-ws-reg-em}) and with the notation introduced in (\ref{THF-Lr}), we can write
\begin{equation}\label{eq-ws-reg-em-ol}
-\int_{\omega'_T}\mathbf{u}_{\epsilon_m}\cdot\varphi_t\,d\mathbf{x}dt+
\alpha\int_{\omega'_T}|\mathbf{u}_{\epsilon_m}|^{\sigma-2}\mathbf{u}_{\epsilon_m}\cdot\varphi\,d\mathbf{x}dt+ \int_{\omega'_T}\mathbf{Q}_{\epsilon_m}:\mathbf{\nabla\varphi}\,d\mathbf{x}dt=0
\end{equation}
for all $\varphi\in\mathbf{C}^{\infty}_0(\omega'_T)$ with $\mathrm{div\,}\varphi=0$ and where $\omega'_T:=\omega'\times(0,T)$. %
The results (\ref{eq-ws-reg-em-ol}) and (\ref{u-wcont})  allow us to apply Proposition~\ref{prop1-dec-p} with $\mathbf{Q}=\mathbf{Q}_{\epsilon_m}$, $\mathbf{q}=\alpha|\mathbf{u}_{\epsilon_m}|^{\sigma-2}\mathbf{u}_{\epsilon_m}$,
$s_1=r$, $s_2=\sigma'$ and
\begin{equation}\label{rr-0}
s_0=r_0:=\min\{r,\sigma'\}.
\end{equation}
Observe that by (\ref{r}), $r\leq 2$ and consequently $r_0\leq 2$. %
Therefore we can say that exist unique functions
\begin{eqnarray}
  & &p^0_{\epsilon_m}\in\mathrm{L}^{r_0}(0,T;\mathrm{A}^{r_0}(\omega')),\label{p0-set-A}\\
  & &\tilde{p}^h_{\epsilon_m}\in\mathrm{C}_{\mathrm{w}}([0,T];\mathrm{\dot{B}}^{r_0}(\omega'))\label{ph-set-Bpt},
\end{eqnarray}
such that
\begin{equation}\label{eq-ws-reg-em-ath}
\begin{split}
& -\int_{\omega'_T}\mathbf{u}_{\epsilon_m}\cdot\varphi_t\,d\mathbf{x}dt+
\alpha\int_{\omega'_T}|\mathbf{u}_{\epsilon_m}|^{\sigma-2}\mathbf{u}_{\epsilon_m}\cdot\mathbf{\varphi}\,d\mathbf{x}dt+ \int_{\omega'_T}\mathbf{Q}_{\epsilon_m}:\mathbf{\nabla}\mathbf{\varphi}\,d\mathbf{x}dt=\\
& \int_{\omega'_T}p^0_{\epsilon_m}\mathrm{div}\mathbf{\varphi}\,d\mathbf{x}dt-
  \int_{\omega'_T}\tilde{p}^h_{\epsilon_m}\frac{\partial\mathrm{div}\mathbf{\varphi}}{\partial\,t}\,d\mathbf{x}dt+
  \int_{\omega'}\mathbf{u}_0\cdot\mathbf{\varphi}(0)\,dt
\end{split}
\end{equation}
for all $\varphi\in\mathbf{C}^{\infty}(\omega'_T)$, with $\mathrm{supp\,}\varphi\subset\subset\omega'\times[0,T)$. %
In addition, by the same result, the following estimates hold
\begin{equation}\label{b-p0-em}
\|p^0_{\epsilon_m}\|_{\mathbf{L}^{r_0}(\omega'_T)}\leq
C_1\left(\|\mathbf{Q}_{\epsilon_m}\|_{\mathbf{L}^r(\omega'_T)}+\|\mathbf{u}_{\epsilon_m}\|_{\mathbf{L}^{\sigma'}(\omega'_T)}\right),
\end{equation}
\begin{equation}\label{b-ph-em}
\|\tilde{p}^h_{\epsilon_m}\|_{\mathrm{L}^{\infty}(0,T;\mathbf{L}^{r_0}(\omega'))}\leq C_2
\left(
\|\mathbf{u}_{\epsilon_m}\|_{\mathrm{L}^{\infty}(0,T;\mathbf{H})}+
\|\mathbf{Q}_{\epsilon_m}\|_{\mathbf{L}^r(\omega'_T)}+\|\mathbf{u}_{\epsilon_m}\|_{\mathbf{L}^{\sigma'}(\omega'_T)}\right),
\end{equation}
where $C_1$ and $C_2$ are positive constants depending only on $q$, $\sigma'$, $N$ and $\omega'_T$. %
Then, from (\ref{p0-set-A}) and (\ref{ph-set-Bpt}) and by means of reflexivity, we get, passing to a subsequence if needed, that
\begin{equation}\label{conv-p0-em}
p^0_{\epsilon_m}\to \underline{p}^0\quad\mbox{in}\quad\mathrm{L}^r(0,T;\mathrm{A}^{r_0}(\omega')),\quad\mbox{as}\quad m\to\infty,
\end{equation}
\begin{equation}\label{conv-ph-em}
\tilde{p}^h_{\epsilon_m}\to \underline{\tilde{p}}^h\quad\mbox{in}\quad\mathrm{L}^{N+1}(0,T;\mathrm{\dot{B}}^{r_0}(\omega')),\quad\mbox{as}\quad m\to\infty.
\end{equation}
Now we pass to the limit $m\to\infty$ in (\ref{eq-ws-reg-em-ath}) by using the convergence results (\ref{convg-La}), (\ref{convg-S-b'}), (\ref{convg-sig-b'}) and (\ref{convg-Phiuxu}), together with the identities $\widetilde{\mathbf{u}}=|\mathbf{u}|^{\sigma-2}\mathbf{u}$ and
$\mathbf{G}=\mathbf{u}\otimes\mathbf{u}$, and also the convergence results (\ref{conv-p0-em}) and (\ref{conv-ph-em}). %
Then we compare this limit equation with that one resulting from applying (\ref{eq-ws-reg-em-ath}) to the limit equation (\ref{limit-eq-SH}), considered for all $\varphi\in\mathbf{C}^{\infty}_0(\omega'_T)$, with $\mathrm{div\,}\varphi=0$, and with the aforementioned identities. %
This procedure yields the existence of unique functions
$p^0$ and $\tilde{p}^h$ satisfying to (\ref{p0-set-A})-(\ref{b-ph-em}). %
Then, since $p^0$ and $\tilde{p}^h$ are uniquely defined, we see that
$$\underline{p}^0=p^0\quad\mbox{and}\quad\underline{\tilde{p}}^h=\tilde{p}^h.$$ %
Proceeding as in \cite[p. 126]{Wolf-2007}, letting $a$ be such that $1<a<\infty$, using the well-known local regularity theory, the compact imbedding $\mathbf{W}^{3,a}(\omega)\hookrightarrow\mathbf{W}^{2,a}(\omega)$ and Lebesgue's theorem of dominated convergence, we can prove that
\begin{equation}\label{conv-ph-em-str}
\tilde{p}^h_{\epsilon_m}\to \tilde{p}^h\quad\mbox{strongly in}\quad\mathrm{L}^{a}(0,T;\mathbf{W}^{2,a}(\omega)),\quad\mbox{as}\quad m\to\infty,
\end{equation}
where $\omega$ is a fixed but arbitrary open bounded subset of $\Omega$ such that
\begin{equation}\label{om-fi}
\omega\subset\subset\omega'\subset\subset\Omega,\quad\mbox{with}\ \partial\omega\ \mbox{Lipschitz}.
\end{equation}
Let us set now
\begin{equation}\label{v-em}
\mathbf{v}_{\epsilon_m}:=\mathbf{u}_{\epsilon_m}+\mathbf{\nabla}\tilde{p}^h_{\epsilon_m},
\end{equation}
\begin{equation}\label{v}
\mathbf{v}:=\mathbf{u}+\mathbf{\nabla}\tilde{p}^h.
\end{equation}
Then, combining (\ref{eq-ws-reg-em-ath}) with this same equation when we pass it to the limit $m\to\infty$, and using the definition of the distributive time derivative, we obtain
\begin{equation}\label{dist-t-der}
\begin{split}
   (\mathbf{v}_{\epsilon_m}-\mathbf{v})'= &\ \mathbf{div}\,(|\mathbf{\nabla}\mathbf{u}_{\epsilon_m}|^{q-2}\mathbf{\nabla}\mathbf{u}_{\epsilon_m}-\mathbf{S})-
   \alpha\left(|\mathbf{u}_{\epsilon_m}|^{\sigma-2}\mathbf{u}_{\epsilon_m}-|\mathbf{u}|^{\sigma-2}\mathbf{u}\right) \\
  & -
\mathbf{div}(\mathbf{u}_{\epsilon_m}\otimes\mathbf{u}_{\epsilon_m}\Phi_{\epsilon_m}(|\mathbf{u}_{\epsilon_m}|)-
\mathbf{u}\otimes\mathbf{u})-
\mathbf{\nabla}(p^0_{\epsilon_m}-p^0)
\end{split}
\quad\mbox{in $\mathcal{D}'(\omega'_T)$}
\end{equation}
Proceeding as for (\ref{xi-ve'}), attending to (\ref{conv-p0-em}) and observing that $r_0\leq r$, we can prove that $$(\mathbf{v}_{\epsilon_m}-\mathbf{v})'\in\mathrm{L}^{r_0}(0,T;\mathbf{W}^{-1,r_0}(\omega'))+\mathbf{L}^{\sigma'}(\omega'_T).$$
Now we shall decompose the pressure term $(p^0_{\epsilon_m}-p^0)$ in (\ref{dist-t-der}) into three new functions. %
For that, we need to invoke the following results, whose proofs follow immediately from \cite[Lemmas 2.3 and 2.4]{Wolf-2007}.

\begin{lemma}\label{L-2.4-W}
Let $1<s<\infty$ and $k\in\mathds{N}$.
\begin{enumerate}
\item Then for every $v^{\ast}\in\left(\mathrm{W}_0^{k,s'}(\omega)\right)'$ there exists a unique $v\in\mathrm{W}_0^{k,s}(U)$ such that
\begin{equation*}
\int_{\omega}\rm{D}^{\alpha}v\,\rm{D}^{\alpha}\varphi\,d\mathbf{x}=\langle v^{\ast},\varphi\rangle\quad
\forall\
\varphi\in\mathrm{C}_0^{\infty}(\omega),\quad |\alpha|=k.
\end{equation*}
\item In addition, if exists $\mathcal{H}\in\mathrm{L}^s(\omega)$ such that
\begin{equation*}
\langle v^{\ast},\varphi\rangle=\int_{\omega}\mathcal{H}\rm{D}^{\alpha}\varphi\,d\mathbf{x}\quad
\forall\
\varphi\in\mathrm{C}_0^{\infty}(\omega),\quad |\alpha|=k,
\end{equation*}
then
\begin{equation*}
\|\rm{D}^{\alpha}v\|_{\mathrm{L}^s(\omega)}\leq C\|\mathcal{H}\|_{\mathrm{L}^s(\omega)},
\end{equation*}
where $C$ is a positive constant depending on $s$ and on the Calderón-Zigmund inequality's constant.
\end{enumerate}
\end{lemma}
\noindent By a direct application of the first part of Lemma~\ref{L-2.4-W}, attending to (\ref{convg-S-b'}), (\ref{convg-sig-b'}) and (\ref{convg-Phiuxu}), and to the definitions of $\mathrm{A}^{q'}(\omega')$ and $\mathrm{A}^{q\frac{N+2}{2N}}(\omega')$, there exist unique functions
\begin{equation}\label{E1-p1}
p^1_{\epsilon_m}\in\mathrm{L}^{q'}(0,T;\mathrm{A}^{q'}(\omega')),
\end{equation}
\begin{equation}\label{E1-p2}
p^2_{\epsilon_m}\in\mathrm{L}^{q\frac{N+2}{2N}}(0,T;\mathrm{A}^{q\frac{N+2}{2N}}(\omega')),
\end{equation}
\begin{equation}\label{E1-p3}
p^3_{\epsilon_m}\in\mathrm{L}^{\sigma'}(0,T;\mathrm{W}^{1,\sigma'}_0(\omega')),
\end{equation}
such that
\begin{equation}\label{dec-p1}
\int_{\omega'_T}p^1_{\epsilon_m}\triangle\phi\,d\mathbf{x}dt=
\int_{\omega'_T}(|\mathbf{\nabla\,u}_{\epsilon_m}|^{q-2}\mathbf{\nabla\,u}_{\epsilon_m}- \mathbf{S}):\nabla^2\phi\,d\mathbf{x}dt,
\end{equation}
\begin{equation}\label{dec-p2}
\int_{\omega'_T}p^2_{\epsilon_m}\triangle\phi\,d\mathbf{x}dt=-
\int_{\omega'_T}(\mathbf{u}_{\epsilon_m}\otimes\mathbf{u}_{\epsilon_m}\Phi_{\epsilon_m}(|\mathbf{u}_{\epsilon_m}|)-
\mathbf{u}\otimes\mathbf{u}):\nabla^2\phi\,d\mathbf{x}dt,
\end{equation}
\begin{equation}\label{dec-p3}
\int_{\omega'_T}p^3_{\epsilon_m}\triangle\phi\,d\mathbf{x}dt=\alpha
\int_{\omega'_T}\left(|\mathbf{u}_{\epsilon_m}|^{\sigma-2}\mathbf{u}_{\epsilon_m}-|\mathbf{u}|^{\sigma-2}\mathbf{u}\right)\cdot\nabla\phi\,d\mathbf{x}dt
\end{equation}
for all $\phi\in\mathrm{C}_0^{\infty}(\omega'_T)$. %
In addition, by (\ref{dec-p1})-(\ref{dec-p3}) and a direct application of the second part of Lemma~\ref{L-2.4-W}, the following estimates hold:
\begin{equation}\label{bd2-p1-em}
\|p^1_{\epsilon_m}\|_{\mathbf{L}^{q'}(\omega'_T)}\leq C_1\||\mathbf{\nabla\,u}_{\epsilon_m}|^{q-2}\mathbf{\nabla\,u}_{\epsilon_m}-\mathbf{S}\|_{\mathbf{L}^{q'}(\omega'_T)};
\end{equation}
\begin{equation}\label{bd2-p2-em}
\|p^2_{\epsilon_m}\|_{\mathbf{L}^{q\frac{N+2}{2N}}(\omega'_T)}\leq C_2\|\mathbf{u}_{\epsilon_m}\otimes\mathbf{u}_{\epsilon_m}\Phi_{\epsilon_m}(|\mathbf{u}_{\epsilon_m}|)-
\mathbf{u}\otimes\mathbf{u}\|_{\mathbf{L}^{q\frac{N+2}{2N}}(\omega'_T)};
\end{equation}
\begin{equation}\label{bd2-p3-em}
\|\mathbf{\nabla} p^3_{\epsilon_m}\|_{\mathbf{L}^{\sigma'}(\omega'_T)}\leq C_3\||\mathbf{u}_{\epsilon_m}|^{\sigma-2}\mathbf{u}_{\epsilon_m}-|\mathbf{u}|^{\sigma-2}\mathbf{u}\|_{\mathbf{L}^{\sigma'}(\omega'_T)};
\end{equation}
where $C_1$, $C_2$ and $C_3$ are positive constants depending on $q'$, $q\frac{N+2}{2N}$ and $\sigma'$, respectively, and on the Calderón-Zigmund inequality's constant ($C_3$ depends also on $\alpha$). %
Next, testing (\ref{dist-t-der}) by $\nabla\phi$, with $\phi\in\mathrm{C}_0^{\infty}(\omega_T)$, integrating over $\omega_T$ and using (\ref{eq2-inc-e}) and (\ref{eq3-pro-dp}) together with (\ref{set-Bpt}), and also the identities (\ref{dec-p1})-(\ref{dec-p3}), we obtain
$$p^0_{\epsilon_m}-p^0=p^1_{\epsilon_m}+p^2_{\epsilon_m}+p^3_{\epsilon_m}\,.$$
Inserting this into (\ref{dist-t-der}), it follows that
\begin{equation}\label{dist2-t-der}
\begin{split}
   (\mathbf{v}_{\epsilon_m}-\mathbf{v})'= &\ \mathbf{div}\left(|\mathbf{\nabla}\mathbf{u}_{\epsilon_m}|^{q-2}\mathbf{\nabla}\mathbf{u}_{\epsilon_m}-\mathbf{S}\right)-
   \alpha\left(|\mathbf{u}_{\epsilon_m}|^{\sigma-2}\mathbf{u}_{\epsilon_m}-|\mathbf{u}|^{\sigma-2}\mathbf{u}\right)\\
  &
-\mathbf{div}\left(\mathbf{u}_{\epsilon_m}\otimes\mathbf{u}_{\epsilon_m}\Phi_{\epsilon_m}(|\mathbf{u}_{\epsilon_m}|)-
\mathbf{u}\otimes\mathbf{u}\right)\\
& -
\mathbf{\nabla}p^1_{\epsilon_m}-\mathbf{\nabla}p^2_{\epsilon_m}-\mathbf{\nabla}p^3_{\epsilon_m}
\end{split}
\quad\mbox{in $\mathcal{D}'(\omega_T)$}.
\end{equation}

\section{Definition of the irregularity regions}\label{Sect-Reg-Irr}

Let us consider the following slight modification of the functions (\ref{v-em})-(\ref{v})
\begin{equation}\label{v-em-chi}
\mathbf{w}_{\epsilon_m}:=(\mathbf{v}_{\epsilon_m}-\mathbf{v})\chi_{\omega_T}\equiv\left(\mathbf{u}_{\epsilon_m}+\mathbf{\nabla}\tilde{p}^h_{\epsilon_m}-\left(\mathbf{u}+\mathbf{\nabla}\tilde{p}^h\right)
\right)\chi_{\omega_T},
\end{equation}
where $\chi_{\omega_T}$ denotes the characteristic function of the set $\omega_T:=\omega\times(0,T)$ and $\omega$ satisfies to (\ref{om-fi}). %
Having in mind the extension of (\ref{dist2-t-der}) to $\mathds{R}^{N+1}$, here we shall consider that
\begin{equation}\label{Ups-1}
\mathbf{Q}^1_{\epsilon_m}:=\mathbf{S}-|\mathbf{\nabla}\mathbf{u}_{\epsilon_m}|^{q-2}\mathbf{\nabla}\mathbf{u}_{\epsilon_m}|+p^1_{\epsilon_m}\mathbf{I},
\end{equation}
\begin{equation}\label{Ups-2}
\mathbf{Q}^2_{\epsilon_m}:=\mathbf{u}_{\epsilon_m}\otimes\mathbf{u}_{\epsilon_m}\Phi_{\epsilon_m}(|\mathbf{u}_{\epsilon_m}|)-\mathbf{u}\otimes\mathbf{u}+p^2_{\epsilon_m}\mathbf{I},
\end{equation}
\begin{equation}\label{Ups-3}
\mathbf{q}_{\epsilon_m}:=\alpha\left(|\mathbf{u}_{\epsilon_m}|^{\sigma-2}\mathbf{u}_{\epsilon_m}-|\mathbf{u}|^{\sigma-2}\mathbf{u}\right)+\mathbf{\nabla}p^3_{\epsilon_m},
\end{equation}
are extended from $\omega_T$ to $\mathds{R}^{N+1}$ by zero. %
Now, since $q<q^{\ast}$, we can use (\ref{str-conv-g}), with $s=\gamma=q$, together with (\ref{conv-ph-em-str}), with $a=q$, to prove that
\begin{equation}\label{str-v-em}
\mathbf{w}_{\epsilon_m}\to 0\quad\mbox{strongly in $\mathbf{L}^{q}(\mathds{R}^{N+1})$},\quad \mbox{as}\ m\to\infty.
\end{equation}
Moreover, using (\ref{est-inf-gam}) and again (\ref{conv-ph-em-str}) with $a=q$, we obtain
\begin{equation}\label{un-b-v-em}
\|\nabla\mathbf{w}_{\epsilon_m}\|_{\mathbf{L}^{q}(\mathds{R}^{N+1})}\leq C.
\end{equation}
On the other hand, due to (\ref{convg-S-b'}) and (\ref{bd2-p1-em}), we have
\begin{equation}\label{un-b-T-p1-em}
\|\mathbf{Q}^1_{\epsilon_m}\|_{\mathbf{L}^{q'}(\mathds{R}^{N+1})}\leq C.
\end{equation}
Moreover,   (\ref{unb-uxu-em}) together with (\ref{bd2-p2-em}), and (\ref{str-conv-u-sig'}) together with  (\ref{bd2-p3-em}), justify, respectively, that
\begin{equation}\label{un-b-uxu-p2-em}
\mathbf{Q}^2_{\epsilon_m}\to 0\quad\mbox{strongly in $\mathbf{L}^{\frac{q(N+2)}{2N}}(\mathds{R}^{N+1})$},\quad \mbox{as}\ m\to\infty,
\end{equation}
\begin{equation}\label{un-b-uxu-p3-em}
\mathbf{q}_{\epsilon_m}\to 0\quad\mbox{strongly in $\mathbf{L}^{\sigma'}(\mathds{R}^{N+1})$},\quad \mbox{as}\ m\to\infty.
\end{equation}

In order to define the irregularity regions of the admissible function that we shall test in (\ref{dist2-t-der}), let us set
\begin{equation}\label{f-em}
f_{\epsilon_m}:=\mathcal{M}^{\ast}(|\mathbf{w}_{\epsilon_m}|),
\end{equation}
\begin{equation}\label{g-em}
g_{\epsilon_m}:=\mathcal{M}^{\ast}(|\nabla\mathbf{w}_{\epsilon_m}|) +
\left(\mathcal{M}^{\ast}(|\mathbf{Q}^1_{\epsilon_m}|)\right)^{\frac{1}{q-1}},
\end{equation}
\begin{equation}\label{h-em}
h_{\epsilon_m}:=\left(\mathcal{M}^{\ast}(|\mathbf{Q}^2_{\epsilon_m}|)\right)^{\frac{1}{q-1}},
\end{equation}
\begin{equation}\label{i-em}
i_{\epsilon_m}:=\left(\mathcal{M}^{\ast}(|\mathbf{q}_{\epsilon_m}|)\right)^{\frac{1}{q-1}},
\end{equation}
where $\mathcal{M}^{\ast}:=\mathcal{M}_t\circ\mathcal{M}_{\mathbf{x}}$. Here $\mathcal{M}_t$ and $\mathcal{M}_{\mathbf{x}}$ denote the Hardy-Littlewood maximal operators, which are defined, for some  function $f\in\mathrm{L}^{p}(\mathds{R}^{N+1})$ with $1<p<\infty$, respectively by
\begin{eqnarray*}
  & \displaystyle \mathcal{M}_t(f)(\mathbf{x},t):=&\sup_{0<r<\infty}\frac{1}{2r}\int_{t-r}^{t+r}|f(\mathbf{x},s)|\,ds, \\
  & \displaystyle \mathcal{M}_{\mathbf{x}}(f)(\mathbf{x},t):=&\sup_{0<R<\infty}\frac{1}{\mathcal{L}_N(B_R(\mathbf{x}))}\int_{B_R(\mathbf{x})}|f(\mathbf{y},s)|\,d\mathbf{y},
\end{eqnarray*}
where $B_R(\mathbf{x})$ denotes the ball of $\mathds{R}^N$ centered at $\mathbf{x}$ and with radius $R>0$. %
Then due to the boundedness of the operator $\mathcal{M}^{\ast}$ from $\mathbf{L}^{p}(\mathds{R}^{N+1})$ into $\mathbf{L}^{p}(\mathds{R}^{N+1})$ for any $p>1$ (see \emph{e.g.} Stein~\cite[p. 5]{Stein-1970}), we obtain
\begin{equation}\label{bd-f-em}
   \|f_{\epsilon_m}\|_{\mathbf{L}^{q}(\mathds{R}^{N+1})}\leq C\|\mathbf{w}_{\epsilon_m}\|_{\mathbf{L}^{q}(\mathds{R}^{N+1})},
\end{equation}
\begin{equation}\label{bd-g-em}
   \|g_{\epsilon_m}\|_{\mathbf{L}^{q}(\mathds{R}^{N+1})}\leq
   C_1\|\nabla\mathbf{w}_{\epsilon_m}\|_{\mathbf{L}^{q}(\mathds{R}^{N+1})} +
   C_2\|\mathbf{Q}^1_{\epsilon_m}\|_{\mathbf{L}^{q'}(\mathds{R}^{N+1})}^{\frac{1}{q-1}},
\end{equation}
\begin{equation}\label{bd-h-em}
   \|\mathcal{M}^{\ast}(|\mathbf{Q}^2_{\epsilon_m}|)\|_{\mathbf{L}^{q\frac{N+2}{2N}}(\mathds{R}^{N+1})}\leq C
   \|\mathbf{Q}^2_{\epsilon_m}\|_{\mathbf{L}^{q\frac{N+2}{2N}}(\mathds{R}^{N+1})},
\end{equation}
\begin{equation}\label{bd-i-em}
   \|\mathcal{M}^{\ast}(|\mathbf{q}_{\epsilon_m}|)\|_{\mathbf{L}^{\sigma'}(\mathds{R}^{N+1})}\leq C
   \|\mathbf{q}_{\epsilon_m}\|_{\mathbf{L}^{\sigma'}(\mathds{R}^{N+1})}.
\end{equation}
Next, let $j_{\epsilon_m}$ be anyone of the functions inside the norms on the left-hand sides of (\ref{bd-f-em})-(\ref{bd-i-em}) and let $s$ be the respective Lebesgue exponent. Using (\ref{un-b-v-em})-(\ref{un-b-uxu-p3-em}), (\ref{bd-f-em})-(\ref{bd-i-em}) and arguing as in \cite[p. 31]{DRW-2010}, we obtain for $j\in\mathds{N}$
\begin{equation*}
\|j_{\epsilon_m}\|_{\mathbf{L}^{s}(\mathds{R}^{N+1})}^s\geq 2^j\ln(2)\inf_{2^{2^j}\leq\tau\leq 2^{2^{j+1}}}\tau^{s}\mathcal{L}_{N+1}\left(\left\{(\mathbf{x},t)\in\mathds{R}^{N+1}:|j_{\epsilon_m}|>\tau\right\}\right).
\end{equation*}
As a consequence, there exists $\lambda_{m,j}\in\left[2^{2^j}, 2^{2^{j+1}}\right]$ such that
\begin{equation}\label{Leb-j}
\mathcal{L}_{N+1}\left\{(\mathbf{x},t)\in\mathds{R}^{N+1}:|j_{\epsilon_m}|>\lambda_{m,j}\right\}\leq
C2^{-j}\lambda_{m,j}^{-p}\,\|j_{\epsilon_m}\|_{\mathbf{L}^{p}(\mathds{R}^{N+1})}.
\end{equation}
Let us consider the following subsets of $\mathds{R}^{N+1}$
\begin{eqnarray}
  &F_{m,j}:=&\left\{(\mathbf{x},t):|f_{\epsilon_m}|>\lambda_{m,j}\right\}, \label{F-mj} \\
  &G_{m,j}:=&\left\{(\mathbf{x},t):|g_{\epsilon_m}|>\lambda_{m,j}\right\}, \label{G-mj} \\
  &H_{m,j}:=&\left\{(\mathbf{x},t):|h_{\epsilon_m}|>\lambda_{m,j}\right\}\equiv\left\{(\mathbf{x},t):\mathcal{M}^{\ast}(|\mathbf{Q}^2_{\epsilon_m}|>\lambda_{m,j}^{q-1}\right\}, \label{H-mj} \\
  &I_{m,j}:=&\left\{(\mathbf{x},t):|i_{\epsilon_m}|>\lambda_{m,j}\right\}\equiv\left\{(\mathbf{x},t):\mathcal{M}^{\ast}(|\mathbf{q}_{\epsilon_m}|>\lambda_{m,j}^{q-1}\right\}. \label{I-mj}
\end{eqnarray}
Then, using (\ref{Leb-j}) in each case separately, we obtain
\begin{eqnarray}
  &\mathcal{L}_{N+1}(F_{m,j})\leq &C2^{-j}\lambda_{m,j}^{-q}\,\|f_{\epsilon_m}\|_{\mathbf{L}^{q}(\mathds{R}^{N+1})}, \label{Leb-F} \\
  &\mathcal{L}_{N+1}(G_{m,j})\leq &C2^{-j}\lambda_{m,j}^{-q}\,\|g_{\epsilon_m}\|_{\mathbf{L}^{q}(\mathds{R}^{N+1})}, \label{Leb-G} \\
  &\mathcal{L}_{N+1}(H_{m,j})\leq & C2^{-j}\lambda_{m,j}^{-(q-1)q\frac{N+2}{2N}}\,\|\mathcal{M}^{\ast}(|\mathbf{Q}^2_{\epsilon_m}|)\|_{\mathbf{L}^{q\frac{N+2}{2N}}(\mathds{R}^{N+1})}, \label{Leb-H} \\
  &\mathcal{L}_{N+1}(I_{m,j})\leq & C2^{-j}\lambda_{m,j}^{-(q-1)\sigma'}\,\|\mathcal{M}^{\ast}(|\mathbf{q}_{\epsilon_m}|)\|_{\mathbf{L}^{\sigma'}(\mathds{R}^{N+1})}. \label{Leb-I}
\end{eqnarray}
Now, since $\lambda_{m,j}\in\left[2^{2^j}, 2^{2^{j+1}}\right]$,  we observe that (\ref{str-v-em}), (\ref{bd-f-em}) and (\ref{Leb-F}) on the one hand, (\ref{un-b-uxu-p2-em}), (\ref{bd-h-em}) and (\ref{Leb-H}) on the other, and yet (\ref{un-b-uxu-p3-em}), (\ref{bd-i-em}) and (\ref{Leb-I}) on another one, imply, respectively,
\begin{equation}\label{lsup-FHI}
\limsup_{m\to\infty}\mathcal{L}_{N+1}(F_{m,j})=0,\quad
\limsup_{m\to\infty}\mathcal{L}_{N+1}(H_{m,j})=0,\quad
\limsup_{m\to\infty}\mathcal{L}_{N+1}(I_{m,j})=0.
\end{equation}
Moreover, since $\mathcal{M}^{\ast}$ is subadditive (see \emph{e.g.} Stein~\cite{Stein-1970}), we get from the definitions of $G_{m,j}$, $H_{m,j}$ and $I_{m,j}$ in (\ref{G-mj})-(\ref{I-mj}), using (\ref{Leb-G})-(\ref{Leb-I}) and (\ref{g-em})-(\ref{i-em}), that
\begin{equation}\label{incl-GHI}
G_{m,j}\cup H_{m,j}\cup I_{m,j}\supset\mathcal{O},
\end{equation}
where
\begin{equation*}
\mathcal{O}:= \left\{(\mathbf{x},t)\in\mathds{R}^{N+1}:
\mathcal{M}^{\ast}\left(\left|\nabla\mathbf{w}_{\epsilon_m}\right|\right)+
\rho_{m,j}\left(\mathcal{M}^{\ast}\left(\left|\mathbf{Q}^{1,2}_{\epsilon_m}\right|\right)+
\mathcal{M}^{\ast}\left(\left|\mathbf{q}_{\epsilon_m}\right|\right)\right)<4\lambda_{m,j}\right\},
\end{equation*}
\begin{equation}\label{rho-mj}
\rho_{m,j}:=\lambda_{m,j}^{2-q}
\end{equation}
and, with the notations of (\ref{Ups-1})-(\ref{Ups-2}),
\begin{equation}\label{Upsilon}
\mathbf{Q}^{1,2}_{\epsilon_m}:=\mathbf{Q}_{\epsilon_m}^1+\mathbf{Q}_{\epsilon_m}^2.
\end{equation}
Setting
\begin{equation}\label{E-mj}
E_{m,j}:=\left(F_{m,j}\cup G_{m,j}\cup H_{m,j}\cup I_{m,j}\right)\cap\omega_T,
\end{equation}
we can readily see that due to (\ref{Leb-F})-(\ref{Leb-I}), (\ref{bd-f-em})-(\ref{bd-i-em}) and to
(\ref{str-v-em})-(\ref{un-b-uxu-p3-em}),
\begin{equation}\label{Leb-E}
\mathcal{L}_{N+1}(E_{m,j})<\infty.
\end{equation}
Moreover, due to (\ref{incl-GHI}), we have
\begin{equation}\label{U-cal}
  \left(\mathcal{O}\cup\mathcal{U}\right)\cap\omega_T\subset E_{m,j}\subset \omega_T,
\end{equation}
where here $\mathcal{U}$ is the set $F_{m,j}$ defined in (\ref{F-mj}). %

\section{Construction of a Lipschitz truncation}\label{Sect-Lip-T}
We are now in conditions to define the truncation we shall consider here. %
Let us consider the following family of cubes
\begin{equation}\label{stes-Q}
\mathcal{C}_{r_n}^{\rho_{m,j}}(\mathbf{x}_n,t_n):=
\left\{(\mathbf{y},s)\in\mathds{R}^{N+1}:d_{\rho_{m,j}}\left((\mathbf{x}_n,t_n),(\mathbf{y},s)\right)<r_n\right\},
\end{equation}
where $r_n>0$, $n\in\mathds{N}$ and $d_{\rho_{m,j}}$ is the metric defined by
\begin{equation}\label{metric}
d_{\rho_{m,j}}\left((\mathbf{x}_n,t_n),(\mathbf{y},s)\right):=
\max\left\{|\mathbf{y}-\mathbf{x}_n|,\sqrt{\rho_{m,j}^{-1}|s-t_n|}\right\}.
\end{equation}
By \cite[Theorem VI.1.1]{Stein-1970})), there exists a Whitney covering of $E_{m,j}$ formed by the family of cubes (\ref{stes-Q})-(\ref{metric}) such that
\begin{equation*}\label{W-cov-Emj}
\bigcup_{n\in\mathds{N}}\mathcal{C}_{\frac{1}{2}r_n}^{\rho_{m,j}}(\mathbf{x}_n,t_n)=E_{m,j}.
\end{equation*}
%\begin{equation}\label{Qij-int}
%\mathcal{C}_{s_i}^{\rho_{m,j}}(\mathbf{x}_i,t_i)\cap \mathcal{C}_{r_k}^{\rho_{m,j}}(\mathbf{x}_k,t_k)\neq\emptyset \Rightarrow \frac{1}{2}r_k\leq s_i<2r_k.
%\end{equation}
Moreover, by \cite[Section VI.1.3]{Stein-1970}, there exists a partition of unity $\psi_n$, $n\in\mathds{N}$, associated to the Whitney covering (\ref{stes-Q})-(\ref{metric}) such that
\begin{equation*}
\sum_{k\in\mathds{N}}\psi_k=1\quad\mbox{in}\quad \mathcal{C}_{r_n}^{\rho_{m,j}}(\mathbf{x}_n,t_n).
\end{equation*}
We are now in conditions to define the Lipschitz truncation.
Following ~\cite[Section 3]{DRW-2010} and \cite[Chapter VI]{Stein-1970}, we define the Lipschitz truncation of $\mathbf{w}_{\epsilon_m}$ subordinated to the Whitney covering (\ref{stes-Q})-(\ref{metric}) by
\begin{equation}\label{Lip-trunc}
\mathcal{T}_{m,j}(\mathbf{w}_{\epsilon_m}):=
\left\{
\begin{array}{ll}
  \mathbf{w}_{\epsilon_m} & \mbox{in }\ \omega_T\setminus E_{m,j}\\
  \displaystyle \sum_{n=1}^{\infty}\psi_n\mathbf{w}_{\epsilon_m}\rst{{\mathcal{C}_{r_n}^{\rho_{m,j}}(\mathbf{x}_n,t_n)}} & \mbox{in }\  E_{m,j}\,.
\end{array}\right.
\end{equation}
The idea of this truncation, is to regularize the function $\mathbf{w}_{\epsilon_m}$ by cutting off the regions $E_{m,j}$ of irregularity and then to extend this restricted function by the Whitney covering (\ref{stes-Q})-(\ref{metric}) to the whole domain again. %

Now, let $\xi\in\mathrm{C}_0^{\infty}(\omega_T)$ be a fixed cut-off function  such that $0\leq\xi\leq 1$ in $\omega_T$ and
let us consider the following admissible test function for (\ref{dist2-t-der})
\begin{equation}\label{test-Lip-trunc}
\mathbf{\phi}_{m,j}:=\xi\,\mathcal{T}_{m,j}(\mathbf{w}_{\epsilon_m}).
\end{equation}
In order to establish the main properties of the Lipschitz truncation (\ref{Lip-trunc}) we are interested in, let
\begin{equation}\label{om-xi}
\omega^{\xi}_T:=\mathrm{supp}\,\xi,\quad\mbox{$\xi$ is the cut-off function of (\ref{test-Lip-trunc}).}
\end{equation}
Note that $\omega^{\xi}_T$ is strictly contained in $\omega_T$, because $0\leq\xi\leq 1$ in $\omega_T$.
Let also
$\mathbf{C}_{\rho_{m,j}}^{0,1}(\omega^{\xi}_T)$ be the space of all Lipschitz continuous functions with respect to the metric (\ref{metric}).
From the definition of $\mathbf{w}_{\epsilon_m}$ (see (\ref{v-em-chi})), using (\ref{est-inf-gam}) and (\ref{u-wcont}) together with (\ref{conv-ph-em-str}), with $a=q$, and (\ref{b-ph-em}), we can prove that
\begin{equation*}
\mathbf{w}_{\epsilon_m}\in\mathrm{L}^{\infty}(0,T;\mathbf{L}^2(\omega))\cap\mathrm{L}^{q}(0,T;\mathbf{W}^{1,q}(\omega)).
\end{equation*}
Then, owing to (\ref{lsup-FHI})-(\ref{U-cal}), we can apply directly~\cite[Theorem 3.9, (i)-(iii)]{DRW-2010} to obtain:
\begin{equation}\label{P1-LipT}
\mathcal{T}_{m,j}(\mathbf{w}_{\epsilon_m})\in\mathbf{C}_{\rho_{m,j}}^{0,1}(\omega^{\xi}_T),
\end{equation}
with the norm depending on $N$, $\omega^{\xi}_T$, $\|\mathbf{w}_{\epsilon_m}\|_{\mathbf{L}^1(E_{m,j})}$, $\|\mathbf{w}_{\epsilon_m}\|_{\mathbf{L}^1(\widetilde{\omega}_T)}$, where $\omega^{\xi}_T\subset\subset\widetilde{\omega}_T\subset\subset\omega_T$;
\begin{eqnarray}
  & & \left\|\mathbf{\nabla}\mathcal{T}_{m,j}(\mathbf{w}_{\epsilon_m})\right\|_{\mathbf{L}^{\infty}(\omega^{\xi}_T)}\leq
C\left(\lambda_{m,j}+\rho_{m,j}^{-1}\delta_{\rho_{m,j},\omega^{\xi}_T}^{-N-3}\|\mathbf{w}_{\epsilon_m}\|_{\mathbf{L}^{1}(E_{m,j})}\right); \label{P2-LipT} \\
  & & \left\|\mathcal{T}_{m,j}(\mathbf{w}_{\epsilon_m})\right\|_{\mathbf{L}^{\infty}(\omega^{\xi}_T)}\leq
C\left(1+\rho_{m,j}^{-1}\delta_{\rho_{m,j},\omega^{\xi}_T}^{-N-2}\|\mathbf{w}_{\epsilon_m}\|_{\mathbf{L}^{1}(E_{m,j})}\right); \label{P3-LipT} \\
  & & \begin{split}
& \left\|\left(\mathcal{T}_{m,j}(\mathbf{w}_{\epsilon_m})\right)'\cdot
  \left(\mathcal{T}_{m,j}(\mathbf{w}_{\epsilon_m})-(\mathbf{w}_{\epsilon_m})\right)\right\|_{\mathbf{L}^{1}(\omega^{\xi}_T\cap E_{m,j})}\leq \\
&  C\rho_{m,j}^{-1}\mathcal{L}_{N+1}(E_{m,j})
  \left(\lambda_{m,j}+\rho_{m,j}^{-1}\delta_{\rho_{m,j},\omega^{\xi}_T}^{-N-3}\|\mathbf{w}_{\epsilon_m}\|_{\mathbf{L}^{1}(E_{m,j})}\right)^2.
\end{split} \label{P4-LipT}
\end{eqnarray}
In (\ref{P2-LipT})-(\ref{P4-LipT}) the constants $C$ are distinct and depend only on $N$, and
\begin{equation}\label{delta}
\delta_{\rho_{m,j},\omega^{\xi}_T}:=d_{\rho_{m,j}}(\omega^{\xi}_T,\omega_T)>0\quad\mbox{due to (\ref{om-xi})}.
\end{equation}
Moreover, according to \cite[Lemma 3.5]{DRW-2010} (see also \cite[Section VI.3]{Stein-1970}),
\begin{equation}\label{P5-LipT}
\left\|\mathcal{T}_{m,j}(\mathbf{w}_{\epsilon_m})\right\|_{\mathbf{L}^{s}(\omega_T)}\leq C
\|\mathbf{w}_{\epsilon_m}\|_{\mathbf{L}^{s}(\omega_T)}\quad\forall\ s: 1\leq s\leq\infty,
\end{equation}
where $C$ depends only on $N$.

\section{Convergence of the approximated extra stress tensor}\label{Sect-Conv-EST}

Proceeding as for (\ref{dist2-t-der}), observing that now the functions are zero outside $\omega_T$ and using the notations (\ref{Upsilon}) and (\ref{v-em-chi})-(\ref{Ups-3}), we obtain
\begin{equation}\label{dist21-t-der}
   \mathbf{w}_{\epsilon_m}'=-\mathbf{div}\mathbf{Q}^{1,2}_{\epsilon_m}-\mathbf{q}_{\epsilon_m}
   \quad\mbox{in $\mathcal{D}'(\omega_T)$}
\end{equation}
Here the distributive time derivative $\mathbf{w}_{\epsilon_m}'$ is such that
\begin{equation}\label{eq-Lr}
\mathbf{w}_{\epsilon_m}'\in\mathrm{L}^{r_0}(0,T;\mathbf{W}^{-1,r_0}(\omega))+\mathbf{L}^{\sigma'}(\omega_T),
\end{equation}
where $r_0$ is defined by (\ref{rr-0}). %
In fact, due to (\ref{convg-S-b'}), (\ref{unb-uxu-em}), (\ref{bd2-p1-em}) and (\ref{bd2-p2-em}) on one hand, and due to (\ref{str-conv-u-sig'}) and (\ref{bd2-p3-em}) on the other, we can prove that
\begin{equation}\label{bd-Ups}
\mathbf{Q}^{1,2}_{\epsilon_m}\in\mathbf{L}^{r_0}(\omega_T)\quad\mbox{and}\quad
\mathbf{q}_{\epsilon_m}\in\mathbf{L}^{\sigma'}(\omega_T).
\end{equation}
As a consequence of (\ref{bd-Ups})$_1$, $\mathbf{div}\mathbf{Q}^{1,2}_{\epsilon_m}\in\mathrm{L}^{r_0}(0,T;\mathbf{W}^{-1,r_0}(\omega))$. %
Now, observing that, by virtue of (\ref{P1-LipT})-(\ref{P3-LipT}) and of the definition of $\xi$, our admissible test function, defined in (\ref{test-Lip-trunc}),
$\mathbf{\phi}_{m,j}\in\mathrm{L}^{r_0'}(0,T;\mathbf{W}^{1,r_0'}_0(\omega))\cap\mathbf{L}^{\sigma}(\omega_T)$. %
Then, from (\ref{dist21-t-der}) and (\ref{eq-Lr}), we infer that
\begin{equation}\label{dist31-t-der}
  \int_0^T\langle\mathbf{w}_{\epsilon_m}'(t),\phi_{m,j}(t)\rangle\,dt= \int_{\omega_T}\mathbf{Q}^{1,2}_{\epsilon_m}:\mathbf{\nabla}\mathbf{\phi}_{m,j}\,d\mathbf{x}dt-
  \int_{\omega_T}\mathbf{q}_{\epsilon_m}\cdot\mathbf{\phi}_{m,j}\,d\mathbf{x}dt.
\end{equation}
On the other hand, owing to (\ref{lsup-FHI})-(\ref{U-cal}) and, in addition, to (\ref{bd-Ups})-(\ref{dist31-t-der}), we can apply \cite[Theorem 3.9, (iv)]{DRW-2010} to prove that for every $\xi\in\mathbf{C}_0^{\infty}(\omega_T)$
\begin{equation}\label{main-r-P2}
\begin{split}
  & \int_0^T\langle\mathbf{w}_{\epsilon_m}'(t),\phi_{m,j}(t)\rangle\,dt=  \\
  & \frac{1}{2}\int_{\omega_T}\left(\left|\mathcal{T}_{m,j}(\mathbf{w}_{\epsilon_m})\right|^2-2\mathbf{w}_{\epsilon_m}\cdot\mathcal{T}_{m,j}(\mathbf{w}_{\epsilon_m})\right)\xi'\,d\mathbf{x}dt+\\
  & \int_{E_{m,j}}\left(\mathcal{T}_{m,j}(\mathbf{w}_{\epsilon_m})\right)'\cdot\left(\mathcal{T}_{m,j}(\mathbf{w}_{\epsilon_m})-\mathbf{w}_{\epsilon_m}\right)\xi\,d\mathbf{x}dt.
\end{split}
\end{equation}
Note that the proof of (\ref{main-r-P2}) is done in \cite[p. 23]{DRW-2010} for $\mathbf{q}_{\epsilon_m}\equiv 0$ in (\ref{dist31-t-der}).
But taking into account (\ref{eq-Lr}), the proof of \cite[Theorem 3.9, (iv)]{DRW-2010} can be repeated almost word by word in our case.

Now, gathering (\ref{dist31-t-der}) and (\ref{main-r-P2}), and expanding the notations (\ref{Upsilon}) and (\ref{v-em-chi})-(\ref{Ups-3}), we obtain
\begin{equation}\label{distpa6-t-der}
\begin{split}
  & \int_{\omega_T}\left(|\mathbf{\nabla}\mathbf{u}_{\epsilon_m}|^{q-2}\mathbf{\nabla}\mathbf{u}_{\epsilon_m}-\mathbf{S}\right):\mathbf{\nabla}(\mathcal{T}_{m,j}(\mathbf{w}_{\epsilon_m}))\,\xi\,d\mathbf{x}dt =\\
    &
  +\int_{\omega_T}\left(\mathbf{S}-|\mathbf{\nabla}\mathbf{u}_{\epsilon_m}|^{q-2}\mathbf{\nabla}\mathbf{u}_{\epsilon_m}\right):\mathcal{T}_{m,j}(\mathbf{w}_{\epsilon_m})\otimes\mathbf{\nabla}\xi\,d\mathbf{x}dt \\
   &
  +\int_{\omega_T}\left(\mathbf{u}_{\epsilon_m}\otimes\mathbf{u}_{\epsilon_m}\Phi_{\epsilon_m}(|\mathbf{u}_{\epsilon_m}|)-\mathbf{u}\otimes\mathbf{u}\right):\mathbf{\nabla}\left(\mathcal{T}_{m,j}(\mathbf{w}_{\epsilon_m})\,\xi\right)\,d\mathbf{x}dt\\ &
  +\alpha\int_{\omega_T}\left(|\mathbf{u}|^{\sigma-2}\mathbf{u}-|\mathbf{u}_{\epsilon_m}|^{\sigma-2}\mathbf{u}_{\epsilon_m}\right)\cdot\mathcal{T}_{m,j}(\mathbf{w}_{\epsilon_m})\,\xi\,d\mathbf{x}dt\\ &
  +\int_{\omega_T}p^1_{\epsilon_m}\mathcal{T}_{m,j}(\mathbf{w}_{\epsilon_m})\cdot\nabla\xi\,d\mathbf{x}dt\\
  &
  +\int_{\omega_T}p^1_{\epsilon_m}\mathrm{div}(\mathcal{T}_{m,j}(\mathbf{w}_{\epsilon_m}))\,\xi\,d\mathbf{x}dt\\
  &
  +\int_{\omega_T}p^2_{\epsilon_m}\mathrm{div}(\mathcal{T}_{m,j}(\mathbf{w}_{\epsilon_m})\,\xi)\,d\mathbf{x}dt\\
  &
  -\int_{\omega_T}\mathbf{\nabla}p^3_{\epsilon_m}\cdot\mathcal{T}_{m,j}(\mathbf{w}_{\epsilon_m})\,\xi\,d\mathbf{x}dt\\
  &
  +\frac{1}{2}\int_{\omega_T}\left(2\mathbf{w}_{\epsilon_m}\cdot\mathcal{T}_{m,j}(\mathbf{w}_{\epsilon_m})-\left|\mathcal{T}_{m,j}(\mathbf{w}_{\epsilon_m})\right|^2\right)\xi'\,d\mathbf{x}dt\\
  &
  +\int_{E_{m,j}}\mathcal{T}_{m,j}'(\mathbf{w}_{\epsilon_m})\cdot\left(\mathbf{w}_{\epsilon_m}-\mathcal{T}_{m,j}(\mathbf{w}_{\epsilon_m})\right)\,\xi\,d\mathbf{x}dt\\
  &:=J_1+J_2+J_3+J_4+J_5+J_6+J_7+J_8+J_9.
  \end{split}
\end{equation}
We claim that, for a fixed $j$,
\begin{equation}\label{distpa7-t-der}
\limsup_{m\to\infty}\left|\int_{\omega_T}\left(|\mathbf{\nabla}\mathbf{u}_{\epsilon_m}|^{q-2}\mathbf{\nabla}\mathbf{u}_{\epsilon_m}-\mathbf{S}\right):\mathbf{\nabla}(\mathcal{T}_{m,j}(\mathbf{w}_{\epsilon_m}))\,\xi\,d\mathbf{x}dt\right|
\leq C2^{-\frac{j}{q}}.
\end{equation}
To prove this, we will carry out the passage to the limit $m\to\infty$ in all absolute values $|J_i|$, $i=1,\dots,9$.

$\bullet\ \limsup_{m\to\infty}(|J_1|+|J_4|)=0$.
Due to (\ref{convg-S-b'}) and (\ref{bd2-p1-em}), $\mathbf{S}-|\mathbf{\nabla}\mathbf{u}_{\epsilon_m}|^{q-2}\mathbf{\nabla}\mathbf{u}_{\epsilon_m}$ and $p^1_{\epsilon_m}$ are uniformly bounded in $\mathbf{L}^{q'}(\omega_T)$. Then, using Hölder's inequality and (\ref{P5-LipT}) together with (\ref{v-em-chi}), led us to
\begin{equation*}
|J_1|+|J_4|\leq C_1\left\|\mathcal{T}_{m,j}(\mathbf{w}_{\epsilon_m})\right\|_{\mathbf{L}^{q}(\omega_T)}
    \leq C_2\left(\|\mathbf{u}_{\epsilon_m}-\mathbf{u}\|_{\mathbf{L}^{q}(\omega_T)}+\|\mathbf{\nabla}(\tilde{p}^h_{\epsilon_m}-\tilde{p}^h)\|_{\mathbf{L}^{q}(\omega_T)}\right).
\end{equation*}
The assertion follows by the application of
(\ref{str-conv-g}) with $s=\gamma=q$ and (\ref{conv-ph-em-str}) with $a=q$, and observing that always $q<q^{\ast}$ for any $q\geq 1$.

$\bullet\ \limsup_{m\to\infty}(|J_2|+|J_6|)=0$. In fact, by Hölder's inequality,
\begin{equation*}
\begin{split}
|J_2|+|J_6|\leq &\,
\|\mathbf{u}_{\epsilon_m}\otimes\mathbf{u}_{\epsilon_m}\Phi_{\epsilon_m}(|\mathbf{u}_{\epsilon_m}|)-\mathbf{u}\otimes\mathbf{u}\|_{\mathbf{L}^{1}(\omega_T)}
\|\mathbf{\nabla}\left(\mathcal{T}_{m,j}(\mathbf{w}_{\epsilon_m})\,\xi\right)\|_{\mathbf{L}^{\infty}(\omega^{\xi}_T)} +\\
&
\,\|p^2_{\epsilon_m}\|_{\mathbf{L}^{1}(\omega_T)}
\|\mathrm{div}\left(\mathcal{T}_{m,j}(\mathbf{w}_{\epsilon_m})\,\xi\right)\|_{\mathbf{L}^{\infty}(\omega^{\xi}_T)}.
\end{split}
\end{equation*}
Then, using Hölder's inequality again and due to (\ref{unb-uxu-em}) and (\ref{bd2-p2-em}), we get\break $\limsup_{m\to\infty}(|J_2|+|J_6|)=0$ if both second multiplying terms on the right-hand side of the above inequality are finite. Indeed, by the application of (\ref{P2-LipT}) and (\ref{P3-LipT}) together with (\ref{v-em-chi}), we get
\begin{equation*}
\begin{split}
&\|\mathbf{\nabla}\left(\mathcal{T}_{m,j}(\mathbf{w}_{\epsilon_m})\,\xi\right)\|_{\mathbf{L}^{\infty}(\omega^{\xi}_T)}\leq \\
& \|\mathbf{\nabla}\mathcal{T}_{m,j}(\mathbf{w}_{\epsilon_m})\|_{\mathbf{L}^{\infty}(\omega^{\xi}_T)}+
C\|\mathcal{T}_{m,j}(\mathbf{w}_{\epsilon_m})\|_{\mathbf{L}^{\infty}(\omega^{\xi}_T)}\leq\\
&
C_1\left(\lambda_{m,j}+\frac{\|\mathbf{v}_{\epsilon_m}-\mathbf{v}\|_{\mathbf{L}^{1}(E_{m,j})}}{\rho_{m,j}\delta_{\rho_{m,j},\omega^{\xi}_T}^{N+3}}\right)+
C_2\left(1+\frac{\|\mathbf{v}_{\epsilon_m}-\mathbf{v}\|_{\mathbf{L}^{1}(E_{m,j})}}{\rho_{m,j}\delta_{\rho_{m,j},\omega^{\xi}_T}^{N+2}}\right).
\end{split}
\end{equation*}
From (\ref{convg-La}) and (\ref{conv-ph-em-str}), the last with $a=q$, $\mathbf{v}_{\epsilon_m}-\mathbf{v}$ is uniformly bounded in
$\mathbf{L}^{1}(E_{m,j})$. %
On the other hand, for a fixed $j\in\mathds{N}$, the sequence
$\lambda_{m,j}$ lies in the interval $\left[2^{2^j},2^{2^{j+1}}\right]$ and, as a consequence, the sequence $\rho_{m,j}=\lambda_{m,j}^{2-q}$  is uniformly bounded from above, which, in turn, by (\ref{delta}) and (\ref{metric}), implies
\begin{equation*}
\inf_{m\in\mathds{N}}\delta_{\rho_{m,j},\omega^{\xi}_T}>0.
\end{equation*}
Analogously, we prove that also $\|\mathrm{div}\left(\mathcal{T}_{m,j}(\mathbf{w}_{\epsilon_m})\,\xi\right)\|_{\mathbf{L}^{\infty}(\omega^{\xi}_T)}$ is finite.

$\bullet\ \limsup_{m\to\infty}(|J_3|+|J_7|)=0$. By Hölder´s inequality and (\ref{bd2-p3-em})
\begin{equation*}
\begin{split}
  |J_3|+|J_7|&\leq \left(\alpha\||\mathbf{u}_{\epsilon_m}|^{\sigma-2}\mathbf{u}_{\epsilon_m}-|\mathbf{u}|^{\sigma-2}\mathbf{u}\|_{\mathbf{L}^{1}(\omega_T)}+
\|\mathbf{\nabla} p^3_{\epsilon_m}\|_{\mathbf{L}^{1}(\omega_T)}\right)\|\mathcal{T}_{m,j}(\mathbf{w}_{\epsilon_m})\|_{\mathbf{L}^{\infty}(\omega^{\xi}_T)}  \\
  & \leq C\||\mathbf{u}_{\epsilon_m}|^{\sigma-2}\mathbf{u}_{\epsilon_m}-|\mathbf{u}|^{\sigma-2}\mathbf{u}\|_{\mathbf{L}^{\sigma'}(\omega_T)}\|\mathcal{T}_{m,j}(\mathbf{w}_{\epsilon_m})\|_{\mathbf{L}^{\infty}(\omega^{\xi}_T)}.
\end{split}
\end{equation*}
Arguing as in the previous case, we can show that, for each $j\in\mathds{N}$, $\mathcal{T}_{m,j}(\mathbf{w}_{\epsilon_m})$ is uniformly bounded in $\mathbf{L}^{\infty}(\omega^{\xi}_T)$. %
Then, by the application of (\ref{str-conv-u-sig'}), it follows that $\limsup_{m\to\infty}(|J_3|+|J_7|)=0$.

$\bullet\ \limsup_{m\to\infty}(|J_5|+|J_9|)\leq C 2^{-\frac{j}{q}}$. %
By the definition of the Lipschitz truncation (see (\ref{Lip-trunc})) together with the fact that $\mathrm{div}\mathbf{w}_{\epsilon_m}=0$ (see (\ref{v-em-chi})), we can write
\begin{equation*}
J_5=\int_{\omega_T^{\xi}\cap E_{m,j}}p^1_{\epsilon_m}\mathrm{div}\mathcal{T}_{m,j}(\mathbf{w}_{\epsilon_m})\,d\mathbf{x}dt.
\end{equation*}
Next we use Hölder's inequality, (\ref{convg-S-b'}) and (\ref{bd2-p1-em}) together with (\ref{P2-LipT}). %
Then, arguing as in the cases for $|J_1|+|J_4|$ and $|J_2|+|J_6|$, we have
\begin{equation*}
\begin{split}
  \limsup_{m\to\infty}|J_5|&\leq C_1\limsup_{m\to\infty}\|\nabla\mathcal{T}_{m,j}(\mathbf{w}_{\epsilon_m})\|_{\mathbf{L}^{q}(\omega_T^{\xi}\cap E_{m,j})} \\
  & \leq
  C_1\limsup_{m\to\infty}\mathcal{L}_{N+1}(E_{m,j})^{\frac{1}{q}}\|\nabla\mathcal{T}_{m,j}(\mathbf{w}_{\epsilon_m})\|_{\mathbf{L}^{\infty}(\omega_T^{\xi})} \\
  & \leq
  C_2\limsup_{m\to\infty}\left[\mathcal{L}_{N+1}(E_{m,j})^{\frac{1}{q}}\left(\lambda_{m,j}+\frac{\|\mathbf{v}_{\epsilon_m}-\mathbf{v}\|_{\mathbf{L}^{1}(E_{m,j})}}{\rho_{m,j}\delta_{\rho_{m,j},\omega^{\xi}_T}^{N+3}}\right)\right]
 \\
  & \leq C_2\limsup_{m\to\infty}\left(\mathcal{L}_{N+1}(E_{m,j})^{\frac{1}{q}}\lambda_{m,j}\right).
\end{split}
\end{equation*}
Next, by the definition of $E_{m,j}$ (see (\ref{E-mj})) and using (\ref{lsup-FHI}), (\ref{Leb-G}), (\ref{bd-g-em}) and
(\ref{un-b-v-em})-(\ref{un-b-T-p1-em}) by this order, we get
$\limsup_{m\to\infty}|J_5|\leq C 2^{-\frac{j}{q}}$.

For $J_9$, we have by using (\ref{P4-LipT}) together with the definition of $\rho_{m,j}$ (see (\ref{rho-mj})) and arguing as we did above for $|J_5|$,
\begin{equation*}
\begin{split}
  \limsup_{m\to\infty}|J_9| & \leq  C\limsup_{m\to\infty}\left[\lambda_{m,j}^{q-2}\mathcal{L}_{N+1}(E_{m,j})
  \left(\lambda_{m,j}+\frac{\|\mathbf{v}_{\epsilon_m}-\mathbf{v}\|_{\mathbf{L}^{1}(E_{m,j})}}{\rho_{m,j}\delta_{\rho_{m,j},\omega^{\xi}_T}^{N+3}}\right)^2\right]\\
                            &   \leq C\limsup_{m\to\infty}\left(\lambda_{m,j}^{q-1}\mathcal{L}_{N+1}(E_{m,j})\right) \\
                            & \leq C 2^{-j}.
\end{split}
\end{equation*}
Then observing that $q>1$, it follows that $\limsup_{m\to\infty}|J_9|\leq C 2^{-\frac{j}{q}}$.

\bigskip\noindent
Throughout the above bullets, we have proven the claim (\ref{distpa7-t-der}) is true. %
On the other hand, arguing as we did for $|J_5|$, we can prove also that, for a fixed $j$,
\begin{equation}\label{distpa8-t-der}
\limsup_{m\to\infty}\left|\int_{E_{m,j}}\left(|\mathbf{\nabla}\mathbf{u}_{\epsilon_m}|^{q-2}\mathbf{\nabla}\mathbf{u}_{\epsilon_m}-\mathbf{S}\right):\mathbf{\nabla}\mathcal{T}_{m,j}(\mathbf{w}_{\epsilon_m})\,\xi\,d\mathbf{x}dt\right|
\leq C2^{-\frac{j}{q}}.
\end{equation}
In consequence, from the definition of $\mathcal{T}_{m,j}$ (see (\ref{Lip-trunc})), (\ref{distpa7-t-der}) and (\ref{distpa8-t-der}), we prove that
\begin{equation}\label{distpa9-t-der}
\limsup_{m\to\infty}\left|\int_{\omega_T\setminus E_{m,j}}\left(|\mathbf{\nabla}\mathbf{u}_{\epsilon_m}|^{q-2}\mathbf{\nabla}\mathbf{u}_{\epsilon_m}-\mathbf{S}\right):\mathbf{\nabla}\mathbf{w}_{\epsilon_m}\,\xi\,d\mathbf{x}dt\right|
\leq C2^{-\frac{j}{q}}.
\end{equation}
Using the definition of $\mathbf{w}_{\epsilon_m}$ (see (\ref{v-em-chi})) and the strong convergence property of $\tilde{p}^h_{\epsilon_m}$ (see (\ref{conv-ph-em-str})), it can be derived from (\ref{distpa9-t-der}) that
\begin{equation}\label{distpa10-t-der}
\limsup_{m\to\infty}\left|\int_{\omega_T\setminus E_{m,j}}\left(|\mathbf{\nabla}\mathbf{u}_{\epsilon_m}|^{q-2}\mathbf{\nabla}\mathbf{u}_{\epsilon_m}-\mathbf{S}\right):\mathbf{\nabla}(\mathbf{u}_{\epsilon_m}-\mathbf{u})\,\xi\,d\mathbf{x}dt\right|
\leq C2^{-\frac{j}{q}}.
\end{equation}
Now, by (\ref{distpa10-t-der}) and (\ref{lsup-FHI}), for each $j\in\mathds{N}$ we can find  a number $m_j\in\mathds{N}$ such that
\begin{eqnarray}
  & \label{est1-subs-mj}& \displaystyle
\left|\int_{\omega_T\setminus E_{m_j,j}}\left(|\mathbf{\nabla}\mathbf{u}_{\epsilon_{m_j}}|^{q-2}\mathbf{\nabla}\mathbf{u}_{\epsilon_{m_j}}-\mathbf{S}\right):\mathbf{\nabla}(\mathbf{u}_{\epsilon_{m_j}}-\mathbf{u})\,\xi\,d\mathbf{x}dt\right|
\leq C2^{-\frac{j}{q}}, \\
  & \label{est2-subs-Fmj}&
\mathcal{L}_{N+1}(F_{m_j,j})\leq C2^{-j}, \\
  & \label{est2-subs-Hmj}& \mathcal{L}_{N+1}(H_{m_j,j})\leq C2^{-j}, \\
  & \label{est2-subs-Imj} &\mathcal{L}_{N+1}(I_{m_j,j})\leq C2^{-j}.
\end{eqnarray}
Setting $\xi_j:=\xi\chi_{\omega_T\setminus E_{m_{j},j}}$, where $\chi_{\omega_T\setminus E_{m_{j},j}}$ denotes the characteristic function of the set  $\omega_T\setminus E_{m_{j},j}$, it can be proved (\emph{cf.}~\cite[pp. 36-37]{DRW-2010}), using (\ref{Leb-G}), (\ref{est2-subs-Fmj})-(\ref{est2-subs-Imj}) and the fact that $\lambda_{m_j,j}\geq 1$, that
\begin{equation}\label{con-xi-j}
\xi_j\to \xi\quad\mbox{a.e. in}\quad \omega_T\quad\mbox{as}\ j\to\infty.
\end{equation}
From (\ref{con-xi-j}), (\ref{convg-La}) and (\ref{convg-S-b'}), we have, by appealing to Lebesgue's theorem of dominated convergence, that
\begin{eqnarray}
  &  \mathbf{\nabla}\mathbf{u}\,\xi_j\to\mathbf{\nabla}\mathbf{u}\,\xi\quad\mbox{strongly in}\quad\mathbf{L}^{q}(\omega_T),\quad\mbox{as}\ j\to\infty, & \label{str1-f}\\
  &  \mathbf{S}\,\xi_j\to\mathbf{S}\,\xi\quad\mbox{strongly in}\quad\mathbf{L}^{q'}(\omega_T),\quad\mbox{as}\ j\to\infty. \label{str2-f} &
\end{eqnarray}
Then, from (\ref{est1-subs-mj}) and (\ref{str1-f})-(\ref{str2-f}), and appealing once more to (\ref{convg-La}), (\ref{convg-S-b'}) and Lebesgue's theorem of dominated convergence, we obtain
\begin{equation}\label{est1-lim-mj}
\lim_{j\to\infty}
\int_{\omega_T}|\mathbf{\nabla}\mathbf{u}_{\epsilon_{m_j}}|^{q}\,\xi_j\,d\mathbf{x}dt=
\int_{\omega_T}\mathbf{S}:\mathbf{\nabla}\mathbf{u}\,\xi\,d\mathbf{x}dt.
\end{equation}
Finally, taking into account (\ref{convg-La}), (\ref{convg-S-b'}), (\ref{con-xi-j}) and (\ref{est1-lim-mj}), we can apply the local Minty trick (\emph{cf.} \cite[Lemma A.2]{Wolf-2007}) to establish that $\mathbf{S}\,\xi=|\mathbf{\nabla}\mathbf{u}|^{q-2}\mathbf{\nabla}\mathbf{u}\,\xi$ a.e. in $\omega_T$. Due to the arbitrariness of $\xi$, $\mathbf{S}=|\mathbf{\nabla}\mathbf{u}|^{q-2}\mathbf{\nabla}\mathbf{u}$ a.e. in $\omega_T$ and the proof of Theorem~\ref{th-exst-wsqv} is concluded. $\hfill \blacksquare$

\section{Remarks}\label{Sect-Rem}

In Theorem~\ref{th-exst-wsqv} we have proved the existence of weak solutions, in the sense of Definition~\ref{weak-sol-vq}, to the problem (\ref{geq1-inc})-(\ref{geq1-bc}) for any
\begin{equation*}
q>\frac{2N}{N+2}\qquad\mbox{and}\qquad \sigma>1.
\end{equation*}
It is only left open the case of $1<q\leq\frac{2N}{N+2}$ for $N\geq 3$. %
But with the methods at our present disposal it seems to be very difficult to prove this case, because the compact imbedding
$\mathrm{W}^{1,q}(\Omega)\hookrightarrow\hookrightarrow \mathrm{L}^2(\Omega)$, which holds only for $q>\frac{2N}{N+2}$, is fundamental in many steps of our proof.

The result established in Theorem~\ref{th-exst-wsqv} is still valid if we consider an extra stress tensor with a $q$-structure satisfying to general growth and coercivity conditions. %
Indeed the proof still holds with minor changes if we assume that the diffusion term $|\mathbf{\nabla}\mathbf{u}|^{q-2}\mathbf{\nabla}\mathbf{u}$ in
(\ref{geq1-vel}) is replaced by a tensor $\mathbf{T}\equiv\mathbf{T}(\mathbf{x},t,\mathbf{D})$ ($\mathbf{D}$ is the symmetric part of $\mathbf{\nabla}\mathbf{u}$) satisfying to
\begin{itemize}
\item $|\mathbf{T}(\mathbf{x},t,\mathbf{A})|\leq C_1|\mathbf{A}|^{q-1}+f_1$ for all
$\mathbf{A}$ in $\mathds{M}^n_{\mathrm{sym}}$, for a.a. $(\mathbf{x},t)$ in $Q_T$ and for any function $f_1$ in $\mathrm{L}^{q'}(Q_T)$ with $f_1\geq 0$;
\item $\mathbf{T}(\mathbf{x},t,\mathbf{A}):\mathbf{A}\geq C_2|\mathbf{A}|^{q}-f_2$ for all
$\mathbf{A}$ in $\mathds{M}^n_{\mathrm{sym}}$, for a.a. $(\mathbf{x},t)$ in $Q_T$ and for any function $f_2$ in $\mathrm{L}^{1}(Q_T)$ with $f_2\geq 0$;
\end{itemize}
where $C_1$ and $C_2$ denote positive constants and
$\mathds{M}^n_{\mathrm{sym}}$ is the vector space of all symmetric $n\times n$ matrices, which is equipped with the scalar product $\mathbf{A}:\mathbf{B}$ and
norm $|\mathbf{A}|=\sqrt{\mathbf{A}:\mathbf{A}}$. %

It is possible to consider unbounded domains with no restriction on the size and shape of $\Omega$. In this case, proceeding as in~\cite[Section 3]{Wolf-2007}, we can prove the regularized problem (\ref{eq2-inc-e})-(\ref{eq1-bc-u-e}) has a unique weak solution for such $\Omega$. %
As a consequence the original problem has a solution for these domains as well.

The uniqueness of weak solutions is, as is well known, an open problem for the generalized Navier-Stokes problem (without damping) for values of $q\leq 2$. By adapting \cite[Théorème 2.5.2]{Lions-1969}, we can prove the weak solution to the problem (\ref{geq1-inc})-(\ref{geq1-bc}) is unique under more restrictive conditions that we have needed to prove the existence. In fact, assuming that $q\geq\frac{N+2}{N}$, replacing the diffusion term $\mathbf{div}\left(|\mathbf{\nabla}\mathbf{u}|^{q-2}\mathbf{\nabla}\mathbf{u}\right)$ in (\ref{geq1-vel}) by $\mathbf{div}\left(|\mathbf{\nabla}\mathbf{u}|^{q-2}\mathbf{\nabla}\mathbf{u}\right)+\mathbf{\triangle}\mathbf{u}$ and having in mind the damping term satisfies to (\ref{gin-xi-eta}), it is possible to prove the uniqueness of weak solution to this modified problem (\ref{geq1-inc})-(\ref{geq1-bc}) (see \cite[Theorem 2]{AMFM-2010}).

A completely different issue, is the important question about the qualitative properties of the weak solutions to the problem (\ref{geq1-inc})-(\ref{geq1-bc}). In this scope we are mainly interested in the extinction in a finite time, once that
the confinement of the weak solutions in a space domain is a much more delicate matter and remains an open problem, with the exceptions of the stationary Stokes and Navier-Stokes problems (\emph{cf.}~\cite{CRM-2002}--\cite{AA-2010} and \cite{AMFM-2010}). %
Proceeding as in \cite[Theorem 3] {AMFM-2010}, letting $\mathbf{u}$ be a weak solution to the problem (\ref{geq1-inc})-(\ref{geq1-bc}) in the sense of Definition~\ref{th-exst-wsqv} and assuming that (\ref{eq1-vel-u0}) and one of the following conditions hold:
\begin{enumerate}
  \item $q<2$; or
  \item $1<\sigma<2$;
\end{enumerate}
then we can prove the following assertions are true with minor changes in the proofs:
\begin{itemize}
  \item If $\textbf{f}=\textbf{0}$ a.e. in $Q_T$,  then there exists
$t^{\ast}_{(i)}>0$ such that $\textbf{u}(\textbf{x},t)=\textbf{0}$ a.e. in $\Omega$ and for all $t\geq t^\ast_{(i)}$, for $i=1,2$;
  \item Let $\textbf{f}\not=\textbf{0}$ and assume that exist positive constants $\epsilon_{(i)}$ and (positive) times $t^{\mathbf{f}}_{(i)}$, for $i=1,\ 2$, such that, for almost all $t\in[0,T]$,
\begin{equation*}\label{cond-f-++}
\|\textbf{f}(t)\|_{\mathbf{L}^{q'}(\Omega)} \leq \epsilon_{(1)} %
\left(1-\frac{t}{{t}^{\mathbf{f}}_{(1)}}\right)_{+}^{\theta_{(1)}}
\quad\mbox{if}\quad \frac{Nq}{N-q}\leq q<2\,,
\quad
\theta_{(1)}=\frac{q-1}{2-q},
\end{equation*}
or %
\begin{equation*}\label{cond-f-+}
\|\textbf{f}(t)\|_{\mathbf{L}^{q'}(\Omega)} \leq \epsilon_{(2)} %
\left(1-\frac{t}{{t}^{\mathbf{f}}_{(2)}}\right)_{+}^{\theta_{(2)}}
\quad\mbox{if}\quad 1<\sigma<2\,,\quad
\theta_{(2)}=\frac{(q-1)[q(N+\sigma)-N\sigma]}{q^2(2-\sigma)}.
\end{equation*}
Then there, there  exist positive constants $\epsilon^{0}_{(i)}$ such
that $\textbf{u}=\textbf{0}$ a.e. in $\Omega$ and for all $t\geq t^{\textbf{f}}_{(i)}$ provided $0<\epsilon_{(i)}\leq\epsilon^{0}_{(i)}$, for $i=1,2$.
\end{itemize}
Note that the subscripts $_{(i)}$ are used above in the sense to relate each result to the different condition $(i)$ written before.

\end{document}